\documentclass[11pt,bookmark=true]{article}

 \RequirePackage{geometry}
\usepackage{graphicx}				

\usepackage{amssymb}
\usepackage{amsmath}
\usepackage{amsthm}
\usepackage{enumerate}
 \usepackage{lineno}

\newtheorem{theorem}{\indent Theorem}[section]

\newtheorem{proposition}[theorem]{\indent Proposition}
\newtheorem{definition}[theorem]{\indent Definition}
\newtheorem{lemma}[theorem]{\indent Lemma}
\newtheorem{remark}[theorem]{\indent Remark}

\begin{document}
\title{Hierarchical exact controllability of the fourth order parabolic equations}
\author{
Bo You\footnote{Email address: youb2013@xjtu.edu.cn}
 \\
{\small School of Mathematics and Statistics, Xi'an Jiaotong University} \\
\small Xi'an, 710049, P. R. China\\
Fang Li\footnote{Email address: fli@xidian.edu.cn}
	 \\
	{\small School of Mathematics and Statistics, Xidian University} \\
	\small Xi'an, 710071, P. R. China}

					
\maketitle

\begin{center}
\begin{abstract}
This paper is concerned with the application of Stackelberg-Nash strategies to control fourth order linear and semi-linear parabolic equations. We assume that the system is acted through a hierarchy of distributed controls: one main control (the leader) that is responsible for an exact controllability property; and a couple of secondary controls (the followers) that minimize two prescribed cost functionals and provides a pair of Nash equilibria for the two prescribed cost functionals.  In this paper, we first prove the existence of an associated Nash equilibrium pair corresponding to a hierarchical bi-objective optimal control problem for each leader by Banach fixed points theorem. Then, we establish an observability inequalities of fourth order coupled parabolic equations by global Carleman inequalities and energy methods. Based on such results, we obtain the existence of a leader that drives the controlled system exactly to a prescribed (but arbitrary) trajectory. Furthermore, we also give the second-order sufficient conditions of optimality for the secondary controls.


\end{abstract}
\end{center}

\section{Introduction}
\def\theequation{1.\arabic{equation}}\makeatother
\setcounter{equation}{0}
The development of science and technology has motivated many branches of control theory. Initially, in the classical control theory, we usually encountered problems where a system must reach a predetermined target by the action of a single control, for example, find the control of minimum norm such that the design specifications are met. To this purpose, one way is to find a control to minimize a cost functional defined on a prescribed set of admissible controls and/or states from the standard optimal control viewpoint. Another way is to seek for a control to reach a desired state at the final time of evolution when we investigate the controllability properties of the controlled system. In many realistic situations, it is necessary to include several different (maybe conflictive or contradictory) control objectives, as well as develop theory that would handle the concepts of multi-criteria optimization, where optimal decisions need to be made in the presence of trade-offs between these different objectives. In particular, the cost functionals may be the sum of several terms and it is not completely clear which is the most suitable average. It can also be expectable to have more than one control acting on the equation. In these cases, we are led to consider multi-objective and/or hierarchic control problems. Moreover, the true goal becomes to find one or several equilibria. Motivated by economics, there exist several equilibrium concepts for multi-objective problems origin in game theory. Each of them determines a strategy. Moreover, these strategies can be cooperative (when the controls mutually cooperate in order to achieve some goals) or noncooperative. It is worthy to mention the noncooperative optimization strategy proposed by Nash \cite{nj}, the Pareto cooperative strategy \cite{pv}, and the Stackelberg hierarchical-cooperative strategy \cite{sh}.

According to the formulation introduced by H. von Stackelberg \cite{sh}, we assume the presence of various local controls, called followers, which have their own objectives, and a global control, called leader, with a different goal from the rest of the players (in this case, the followers). The general idea of this strategy is a game of hierarchical nature, where players compete against each other, so that the leader makes the first move and then the followers react optimally to the action of the leader. Since many followers are present and each has a specific objective, it is intended that these are in Nash equilibrium.

In the context of the control of PDEs, a relevant question is whether one is able to steer the system to a desired state (exactly or approximately) by applying controls that correspond to one of these strategies.
Inspired by ideas of Lions \cite{ljl2, ljl}, there are some works concerning hierarchical control employing the Stackelberg strategy. In particular,  the author in \cite{ljl1} gave some results concerning Pareto and Stackelberg strategies, respectively. The approximate controllability of a system was established following a Stackelberg-Nash strategy in \cite{dji}. Moreover, they in \cite{dji1} also provided a characterization of the solution by means of Fenchel-Rockafellar duality theory. The Nash equilibrium for constraints given by linear parabolic and Burger's equations was analyzed in \cite{ram2, ram1} from the mathematical and numerical viewpoints. In \cite{afd3, afd2, afd, afd1, hsv, lj}, the authors studied the Stackelberg-Nash exact controllability for linear and semilinear parabolic equations.  The existence of a Nash equilibrium for a nonlinear distributed parameter predator-prey system was proved in \cite{ram}.Furthermore, they proposed and analyzed a conceptual approximation algorithm. In \cite{afd4, ggf}, the authors studied a Stackelberg-Nash approximate controllability for two coupled equations of fluid mechanics and the Stokes systems. Recently, in \cite{afd5},the authors also analyzed the hierarchic control for wave equation with distributed leader and follower controls.

Until now, there are some results about the controllability for fourth order parabolic equations in both one dimension (see \cite{cn5, cn7, cn6, ce1, ce, gp, hsv1, lgm}) and the higher dimensions (see \cite{cy, dji2, gs, kk, lq, yh}). In particular, the approximate controllability and non-approximate controllability of higher order parabolic equations were studied in \cite{dji}. The author in \cite{yh} proved the null controllability of fourth order parabolic equations by using the ideas of \cite{lg}. It is worthy to mention that the Carleman inequality for a fourth order parabolic equation with $n\geq 2$ was first established in \cite{gs}. Later, the author in \cite{kk} proved the null controllability and the exact controllability to the trajectories at any time $T>0$ for the fourth order semi-linear parabolic equations with a control function acting at the interior.  Recently, the null controllability for fourth order stochastic parabolic equations was obtained by duality arguments and a new global Carleman estimates in \cite{lq}.  A unified weighted inequality for fourth-order partial differential operators was given in \cite {cy}. Moreover, they applied it to obtain the log-type stabilization result for the plate equation. To the best of our knowledge, there is no result concerning the Stackelberg-Nash exact controllability  of fourth order parabolic equations for $n\geq 2.$

In this paper, we mainly consider some results on the hierarchic exact controllability of the following fourth order parabolic equation through Stackelberg-Nash strategies:
\begin{equation}\label{1.1}
\begin{cases}
\frac{\partial u}{\partial t}+\Delta^2u+a(x,t)u+B(x,t)\cdot\nabla u=F(u,\nabla u)+f\chi_{\mathcal{O}}+v_1\chi_{\mathcal{O}_1}+v_2\chi_{\mathcal{O}_2},\,\,\,\,(x,t)\in Q,\\
u=\frac{\partial u}{\partial\vec{n}}=0,\,\,\,\,\,(x,t)\in\Sigma,\\
u(x,0)=u_0(x),\,\,\,\,\,x\in\Omega,
\end{cases}
\end{equation}
where $\Omega\subset \mathbb{R}^n$ ($n\geq 2$) is a bounded domain with boundary $\Gamma$ of class $\mathcal{C}^4,$ $\vec{n}$ is the outward unit norm vector of $\Gamma$ at the point $x\in\Gamma,$ $\mathcal{O},$ $\mathcal{O}_1,$ $\mathcal{O}_2\subset \Omega$ are three small non-empty open sets, $\chi_\omega$ is the characteristic function of the set $\omega,$ $f,$ $v_1,$ $v_2$ are given in some appropriate spaces, the function $F:\mathbb{R}\times \mathbb{R}^d\rightarrow\mathbb{R}$ is locally Lipschitz continuous, $a\in L^{\infty}(Q),$ $B=(B_1,B_2,\cdots,B_n)\in L^{\infty}(Q;\mathbb{R}^n).$ For any given $T>0,$ denote by $Q=\Omega\times (0,T),$ $Q_\omega=\omega\times (0,T)$ and $\Sigma=\Gamma\times (0,T).$

Let $\mathcal{O}_{1,d},$ $\mathcal{O}_{2,d}$ be nonempty open subsets of $\Omega$ and define the secondary cost functionals
\begin{align}\label{1.2}
J_i(f;v_1,v_2)=\frac{\alpha_i}{2}\int_0^T\int_{\mathcal{O}_{i,d}}|u(x,t)-\zeta_{i,d}(x,t)|^2\,dxdt+\frac{\mu_i}{2}\int_0^T\int_{\mathcal{O}_i}|v_i(x,t)|^2\,dxdt,
\end{align}
where the functions $\zeta_{i,d}$ are given in $L^2(\mathcal{O}_{i,d}\times (0,T))$ and $\alpha_i,$ $\mu_i$ are positive constants for $i=1,2.$ We also introduce the main functional
\begin{align}\label{1.3}
J(f)=\frac{1}{2}\int_0^T\int_{\mathcal{O}}|f(x,t)|^2\,dxdt.
\end{align}
First of all, for any given leader $f,$ we will try to find a couple of controls $v_1$ and $v_2$ depending on $f,$ which "minimize simultaneously" $J_1$ and $J_2$ in the following sense:
\begin{align}\label{1.4}
J_1(f;v_1,v_2)=\min_{\hat{v}_1}J_1(f;\hat{v}_1,v_2),\,\,\,\,\,J_2(f;v_1,v_2)=\min_{\hat{v}_2}J_1(f;v_1,\hat{v}_2).
\end{align}
A pair of such controls $(v_1,v_2)$ satisfying \eqref{1.4} is called a Nash equilibrium for $J_1$ and $J_2$ associated to $f.$

Let $\bar{u}=\bar{u}(x,t)$ be a solution of problem
\begin{equation}\label{1.5}
\begin{cases}
\frac{\partial \bar{u}}{\partial t}+\Delta^2\bar{u}+a(x,t)\bar{u}+B(x,t)\cdot\nabla \bar{u}=F(\bar{u},\nabla \bar{u}),\,\,\,\,(x,t)\in Q,\\
\bar{u}=\frac{\partial\bar{u}}{\partial\vec{n}}=0,\,\,\,\,\,(x,t)\in\Sigma,\\
\bar{u}(x,0)=\bar{u}_0(x),\,\,\,\,\,x\in\Omega.
\end{cases}
\end{equation}
After proving that there exists a Nash equilibrium for each leader $f,$ we will choose a control $f\in L^2(\mathcal{O}\times (0,T)),$ such that
\begin{align}\label{1.6}
J(f)=\min_{\hat{f}}J(\hat{f})
\end{align}
subject to the following exact controllability condition
\begin{align}\label{1.7}
u(x,T)=\bar{u}(x,T),\,\,\,\,\forall\,\,\,x\in\Omega.
\end{align}
Thus, our main goal is to prove that such triplets $(f;v_1,v_2)$ exists.

The rest of this paper is organized as follows. In Section 2, we recall the Carleman inequalities for fourth order parabolic equation with $n\geq 2.$ In Section 3, we first analyze the existence of a Nash equilibrium by Banach's fixed points theorem and prove the exact controllability of problem \eqref{1.1} in the linear case. In Section 4, we establish the exact controllability of problem \eqref{1.1} by using Leray-Schauder fixed-point argument in the semi-linear case. Moreover, we also analyze the relations between Nash equilibrium and Nash quasi-equilibrium.
\section{\bf Preliminaries}
\def\theequation{2.\arabic{equation}}\makeatother
\setcounter{equation}{0}
In this section, we will recall the global Carleman inequalities of fourth order parabolic equations and some lemmas used in the sequel. To this purpose, we need a special weight function.
\begin{lemma}(\cite{fav})\label{2.1}
Let $\omega_0\subset\subset \Omega$ be an arbitrary fixed subdomain of $\Omega.$ Then there exists a function $\eta\in\mathcal{C}^4(\overline{\Omega}),$ such that
\begin{align*}
\eta(x)>0,\,\,\,\,\forall\,\,\,x\in \Omega;\,\,\,\eta(x)=0,\,\,\,\,\forall\,\,\,x\in\Gamma;\,\,\,|\nabla\eta(x)|>0,\,\,\,\,\forall\,\,\,x\in\overline{\Omega\backslash\omega_0}.
\end{align*}
\end{lemma}
In order to state the global Carleman inequalities, we introduce the following two weight functions:
\begin{align}\label{2.2}
\alpha(x,t)=\frac{e^{\lambda(2\|\eta\|_{L^{\infty}(\Omega)}+\eta(x))}-e^{4\lambda\|\eta\|_{L^{\infty}(\Omega)}}}{\sqrt{t(T-t)}},\,\,\,\,\xi(x,t)=\frac{e^{\lambda(2\|\eta\|_{L^{\infty}(\Omega)}+\eta(x))}}{\sqrt{t(T-t)}},
\end{align}
which possess the following properties:
\begin{align}\label{2.0}
\nabla\alpha=\nabla\xi=\lambda\xi\nabla\eta,\,\,\,\,\,\,\,
\xi^{-1}\leq \frac{T}{2},\,\,\,\,\,\,\,
|\alpha_t|+|\xi_t|\leq\frac{T}{2}\xi^3,\,\,\,\,\forall\,\,(x,t)\in Q.
\end{align}
\begin{lemma}\label{2.4}(see \cite{gs})
Assume that $z_0\in L^2(\Omega)$ and $g\in L^2(Q),$ the functions $\alpha,$ $\xi$ are defined by \eqref{2.2}. Then there exists a constant $\hat{\lambda}>0$ such that for an arbitrary $\lambda\geq \hat{\lambda},$ we can choose $s_0=s_0(\lambda)>0$ satisfying: there exists a constant $C=C(\lambda)>0$ independent of $s,$ such that the solution $z\in L^2(Q)$ to problem
\begin{equation*}
\begin{cases}
-\frac{\partial z}{\partial t}+\Delta^2 z=g,\,\,\,\,\forall\,\,\,(x,t)\in Q,\\
z=\frac{\partial z}{\partial\vec{n}}=0,\,\,\,\,\,\,\,\,\,\,\,\,\,\forall\,\,\,\,(x,t)\in\Sigma,\\
z(x,T)=z_0(x),\,\,\,\,\forall\,\,\,\,x\in\Omega
\end{cases}
\end{equation*}
satisfies the following inequality: for any $\lambda\geq \hat{\lambda}$  and any $s\geq s_0(\lambda)(\sqrt{T}+T),$ one has
\begin{align*}
&\int_Q\left(s^6\lambda^8\xi^6|z|^2+s^4\lambda^6\xi^4|\nabla z|^2+s^3\lambda^4\xi^3|\Delta z|^2+s^2\lambda^4\xi^2|\nabla^2z|^2+s\lambda^2\xi|\nabla\Delta z|^2\right)e^{2s\alpha}\,dxdt\\
&\leq C\left(\int_{Q_\omega}s^7\lambda^8\xi^7|z|^2e^{2s\alpha}\,dxdt+\int_Q|g|^2e^{2s\alpha}\,dxdt\right).
\end{align*}
\end{lemma}

\begin{lemma}\label{2.5}(see \cite{kk})
Assume that $z_0\in H^{-2}(\Omega),$ $F_0\in L^2(Q),$ $F_1\in (L^2(Q))^n,$ $\hat{F}=(\hat{F}_{i,j})_{1\leq i,j\leq n}\in (L^2(Q))^{n^2},$ $\tilde{F}=(\tilde{F}_{i,j,k})_{1\leq i,j,k\leq n}\in (L^2(Q))^{n^3},$ $f_0\in L^2(\Sigma),$ $\tilde{f}\in L^2(\Sigma)$ and the functions $\alpha,$ $\xi$ are defined by \eqref{2.2}. Then there exists a constant $\hat{\lambda}>0$ such that for an arbitrary $\lambda\geq \hat{\lambda},$ we can choose $s_0=s_0(\lambda)>0$ satisfying: there exists a constant $C=C(\lambda)>0$ independent of $s,$ such that the solution $z\in L^2(Q)$ to problem
\begin{equation*}
\begin{cases}
-\frac{\partial z}{\partial t}+\Delta^2 z=F_0+\nabla\cdot F_1+\sum_{i,j=1}^n\frac{\partial^2\hat{F}_{ij}}{\partial x_i\partial x_j}+\sum_{i,j,k=1}^n\frac{\partial^3\tilde{F}_{ijk}}{\partial x_i\partial x_j\partial x_k},\,\,\,\,\forall\,\,\,(x,t)\in Q,\\
z=f_0,\,\,\,\frac{\partial z}{\partial\vec{n}}=\tilde{f},\,\,\,\,\,\forall\,\,\,\,(x,t)\in\Sigma,\\
z(x,T)=z_0(x),\,\,\,\,\forall\,\,\,\,x\in\Omega
\end{cases}
\end{equation*}
satisfies the following inequality: for any $\lambda\geq \hat{\lambda}$  and any $s\geq s_0(\lambda)(\sqrt{T}+T),$ one has
\begin{align*}
&\int_Q\left(s^6\lambda^8\xi^6|z|^2+s^4\lambda^6\xi^4|\nabla z|^2+s^3\lambda^4\xi^3|\Delta z|^2+s^2\lambda^4\xi^2|\nabla^2z|^2+s\lambda^2\xi|\nabla\Delta z|^2\right)e^{2s\alpha}\,dxdt\\
\leq &\int_Q\left(|F_0|^2+s^2\lambda^2\xi^2|F_1|^2+s^4\lambda^4\xi^4\sum_{i,j=1}^n|\hat{F}_{ij}|^2+s^6\lambda^6\xi^6\sum_{i,j,k=1}^n|\tilde{F}_{ijk}|^2\right)e^{2s\alpha}\,dxdt\\
&+C\int_\Sigma\left(s^5\lambda^5\xi^3|\tilde{f}|^2e^{2s\alpha}+s^7\lambda^7\xi^7|f_0|^2e^{2\alpha}+s^5\lambda^5\xi^5\sum_{i,j,k=1}^n|\tilde{F}_{ijk}n_j|^2e^{2s\alpha}\right)\,dSdt\\
&+C\int_{Q_\omega}s^7\lambda^8\xi^7|z|^2e^{2s\alpha}\,dxdt.
\end{align*}
\end{lemma}

\section{The linear case}
\def\theequation{3.\arabic{equation}}\makeatother
\setcounter{equation}{0}
In this section, we will study the exact controllability of problem \eqref{1.1} with distributed leader and followers in the case that $F\equiv 0.$ Thanks to the linearity of the problem under consideration, we can reduce the exact controllability of the trajectories to a null controllability property. In fact, if we put $w=u-\bar{u},$ then it is immediate to deduce from problem \eqref{1.1} and problem \eqref{1.5} that $w$ is the solution of the following problem 
\begin{equation}\label{eq2.1}
\begin{cases}
\frac{\partial w}{\partial t}+\Delta^2w+a(x,t)w+B(x,t)\cdot\nabla w=f\chi_{\mathcal{O}}+v_1\chi_{\mathcal{O}_1}+v_2\chi_{\mathcal{O}_2},\,\,\,\,(x,t)\in Q,\\
w=\frac{\partial w}{\partial\vec{n}}=0,\,\,\,\,\,(x,t)\in\Sigma,\\
w(x,0)=w_0(x),\,\,\,\,\,x\in\Omega,
\end{cases}
\end{equation}
where $w_0(x)=u_0(x)-\bar{u}_0(x).$ Moreover, the condition \eqref{1.7} is equivalent to 
\begin{align}\label{eq2.2}
w(x,T)=0,\,\,\,\,\forall\,\,\,x\in\Omega
\end{align}
and we can also rewrite the functionals $J_i$ in \eqref{1.2} in terms of $w:$ 
\begin{align*}
J_i(f;v_1,v_2)=\frac{\alpha_i}{2}\int_0^T\int_{\mathcal{O}_{i,d}}|w(x,t)-w_{i,d}(x,t)|^2\,dxdt+\frac{\mu_i}{2}\int_0^T\int_{\mathcal{O}_i}|v_i(x,t)|^2\,dxdt
\end{align*} 
with $w_{i,d}=\zeta_{i,d}-\bar{u}$ for $i=1,2.$
\subsection{Nash equilibrium}
In this subsection, we will establish a result concerning the existence and uniqueness of a Nash equilibrium in the sense of \eqref{1.4} for any $f\in L^2(\mathcal{O}\times (0,T)),$ and we will also give a result which characterizes this Nash equilibrium in terms of the solution of an adjoint system.

\subsubsection{The existence and uniqueness}
To start with, we define $\mathcal{H}_i=L^2(\mathcal{O}_i\times (0,T))$ for any $i=1,2$ and $\mathcal{H}=\mathcal{H}_1\times \mathcal{H}_2.$ For any control $f,$  we infer from the definition of Nash equilibrium that the pair $(v_1,v_2)$ is a Nash equilibrium if and only if it satisfies $J_i'(f;v_1,v_2)\hat{v}_i=0$ for any $\hat{v}_i\in L^2(\mathcal{O}_i\times (0,T))$ and any $i=1,2,$
 i.e.,
\begin{align}\label{eq2.3}
\alpha_i\int_0^T\int_{\mathcal{O}_{i,d}}(w-w_{i,d})w_i\,dxdt+\mu_i\int_0^T\int_{\mathcal{O}_i}v_i\hat{v}_i\,dxdt=0
\end{align}
for any $\hat{v}_i\in \mathcal{H}_i,$ where $w_i$ is the derivative of $w$ with respect to $v_i$ in the direction $\hat{v}_i,$ i.e., $w_i$ is the solution of problem
\begin{equation}\label{eq2.4}
\begin{cases}
\frac{\partial w_i}{\partial t}+\Delta^2w_i+a(x,t)w_i+B(x,t)\cdot\nabla w_i=\hat{v}_i\chi_{\mathcal{O}_i},\,\,\,\,(x,t)\in Q,\\ 
w_i=\frac{\partial w_i}{\partial\vec{n}}=0,\,\,\,\,\,(x,t)\in\Sigma,\\
w_i(x,0)=0,\,\,\,\,\,x\in\Omega.
\end{cases}
\end{equation}
By the regularity of parabolic equation, we can define the operators $A_i\in L(\mathcal{H}_i,L^2(Q))$ by
\begin{align*}
A_i\hat{v}_i=w_i,
\end{align*}
where $w_i$ is the solution of problem \eqref{eq2.4}. Let $z$ be the solution of problem 
\begin{equation}\label{eq2.5}
\begin{cases}
\frac{\partial z}{\partial t}+\Delta^2z+a(x,t)z+B(x,t)\cdot\nabla z=f\chi_{\mathcal{O}},\,\,\,\,(x,t)\in Q,\\
z=\frac{\partial z}{\partial\vec{n}}=0,\,\,\,\,\,(x,t)\in\Sigma,\\
z(x,0)=w_0,\,\,\,\,\,x\in\Omega
\end{cases}
\end{equation}
and denote by $w=z+A_1v_1+A_2v_2,$ then we can recast equality \eqref{eq2.3} into the following form:
\begin{align}\label{eq2.6}
\alpha_i\int_0^T\int_{\mathcal{O}_{i,d}}(z+A_1v_1+A_2v_2-w_{i,d})A_i\hat{v}_i\,dxdt+\mu_i\int_0^T\int_{\mathcal{O}_i}v_i\hat{v}_i\,dxdt=0
\end{align}
for any $\hat{v}_i\in \mathcal{H}_i,$ which implies that $(v_1,v_2)$ is a Nash equilibrium if and only if
\begin{align}\label{eq2.7}
\alpha_iA_i^*\left((z+A_1v_1+A_2v_2-w_{i,d})\chi_{\mathcal{O}_{i,d}}\right)+\mu_iv_i=0,\,\,\,\textit{in}\,\,\,\,\mathcal{H}_i
\end{align}
for any $i=1,2,$ where $A_i^*$ is the adjoint operator of $A_i.$ Denote by
\begin{align}
A(v_1,v_2)=\left(\alpha_1A_1^*\left((A_1v_1+A_2v_2)\chi_{\mathcal{O}_{1,d}}\right)+\mu_1v_1,\alpha_2A_2^*\left((A_1v_1+A_2v_2)\chi_{\mathcal{O}_{2,d}}\right)+\mu_2v_2\right)
\end{align}
and
\begin{align}
B=\left(\alpha_1A_1^*\left((w_{1,d}-z)\chi_{\mathcal{O}_{1,d}}\right),\alpha_2A_2^*\left((w_{2,d}-z)\chi_{\mathcal{O}_{2,d}}\right)\right),
\end{align}
then  equation \eqref{eq2.7} can be rewritten  into the following functional form:
\begin{align}\label{eq2.8}
A(v_1,v_2)=B, \,\,\textit{in}\,\,\mathcal{H}.
\end{align}
In what follows, we will establish the well-posedness of problem \eqref{eq2.8} by Lax-Milgram Theorem.
\begin{proposition}\label{pro2.1}
Let
\begin{align*}
M_0=\max_{i=1,2}\{\|A_i\|_{L(\mathcal{H}_i,L^2(\mathcal{O}_{i,d}\times (0,T)))},\|A_i\|_{L(\mathcal{H}_i,L^2(\mathcal{O}_{3-i,d}\times (0,T)))}\}.
\end{align*}
Assume that $M_0^2+4\leq 4\frac{\min\{\mu_1,\mu_2\}}{\max\{\alpha_1,\alpha_2\}}.$   Then the operator $A$ is an isomorphism from $\mathcal{H}$ to $\mathcal{H}^*.$ In particular, for any $f\in L^2(\mathcal{O}\times (0,T)),$ there exists a unique Nash equilibrium $(v_1(f),v_2(f))$ of problem \eqref{eq2.8}. Moreover, there exists a generic constant $\mathcal{K}>0$ depending on $\|a\|_{L^{\infty}(Q)},$ $\|B\|_{L^{\infty}(Q)},$ $\alpha_i,$ $M_0,$ $\|w_0\|_{L^2(\Omega)},$ $\|w_{1,d}\|_{L^2(\mathcal{O}_{1,d}\times(0,T))},$ $\|w_{2,d}\|_{L^2(\mathcal{O}_{2,d}\times(0,T))}$ and $T,$ such that
\begin{align}\label{eq2.9}
\|(v_1(f),v_2(f))\|_{\mathcal{H}}\leq\mathcal{K}\left(1+\|f\|_{L^2(\mathcal{O}\times (0,T))}\right).
\end{align}
\end{proposition}
\begin{proof}
First of all, we conclude from H\"{o}lder's inequality that for any $(\hat{v}_1,\hat{v}_2)\in\mathcal{H},$
\begin{align*}
&\langle A(v_1,v_2),(\hat{v}_1,\hat{v}_2)\rangle=\int_0^T\int_{\mathcal{O}_1}\mu_1v_1\hat{v}_1+\alpha_1(A_1v_1+A_2v_2)A_1\hat{v}_1\chi_{\mathcal{O}_{1,d}}\,dxdt\\
&+\int_0^T\int_{\mathcal{O}_2}\mu_2v_2\hat{v}_2+\alpha_2(A_1v_1+A_2v_2)A_2\hat{v}_2\chi_{\mathcal{O}_{2,d}}\,dxdt\\
\leq&\mu_1\|v_1\|_{\mathcal{H}_1}\|\hat{v}_1\|_{\mathcal{H}_1}+\alpha_1(\|A_1v_1\chi_{\mathcal{O}_{1,d}}\|_{\mathcal{H}_1}+\|A_2v_2\chi_{\mathcal{O}_{1,d}}\|_{\mathcal{H}_1})\|A_1\hat{v}_1\chi_{\mathcal{O}_{1,d}}\|_{\mathcal{H}_1}\\
&+\mu_2\|v_2\|_{\mathcal{H}_2}\|\hat{v}_2\|_{\mathcal{H}_2}+\alpha_2(\|A_1v_1\chi_{\mathcal{O}_{2,d}}\|_{\mathcal{H}_2}+\|A_2v_2\chi_{\mathcal{O}_{2,d}}\|_{\mathcal{H}_2})\|A_2\hat{v}_2\chi_{\mathcal{O}_{2,d}}\|_{\mathcal{H}_2}\\
\leq&(\mu_1+\mu_2+2(\alpha_1+\alpha_2)M_0^2)\|(v_1,v_2)\|_{\mathcal{H}}\|(\hat{v}_1,\hat{v}_2)\|_{\mathcal{H}}
\end{align*}
and
\begin{align*}
\langle A(v_1,v_2),(v_1,v_2)\rangle=&\int_0^T\int_{\mathcal{O}_1}\mu_1|v_1|^2+\alpha_1(A_1v_1+A_2v_2)A_1v_1\chi_{\mathcal{O}_{1,d}}\,dxdt\\
&+\int_0^T\int_{\mathcal{O}_2}\mu_2|v_2|^2+\alpha_2(A_1v_1+A_2v_2)A_2v_2\chi_{\mathcal{O}_{2,d}}\,dxdt\\
\geq&\int_0^T\int_{\mathcal{O}_1}\mu_1|v_1|^2-\frac{\alpha_1}{4}|A_2v_2|^2\chi_{\mathcal{O}_{1,d}}\,dxdt\\
&+\int_0^T\int_{\mathcal{O}_2}\mu_2|v_2|^2-\frac{\alpha_2}{4}|A_1v_1|^2\chi_{\mathcal{O}_{2,d}}\,dxdt\\
\geq&(\mu_1-\frac{\alpha_2}{4}M_0^2)\|v_1\|_{\mathcal{H}_1}^2+(\mu_2-\frac{\alpha_1}{4}M_0^2)\|v_2\|_{\mathcal{H}_2}^2\\
\geq&\frac{1}{4}\min\{\alpha_1,\alpha_2\}\|(v_1,v_2)\|_{\mathcal{H}}^2.
\end{align*}
Moreover, for any $(\hat{v}_1,\hat{v}_2)\in\mathcal{H},$ we have
\begin{align*}
\langle B, (\hat{v}_1,\hat{v}_2)\rangle=&\alpha_1\int_0^T\int_{\mathcal{O}_1}(w_{1,d}-z)A_1\hat{v}_1\chi_{\mathcal{O}_{1,d}}\,dxdt+\alpha_2\int_0^T\int_{\mathcal{O}_2}(w_{2,d}-z)A_2\hat{v}_2\chi_{\mathcal{O}_{2,d}}\,dxdt\\
\leq&\alpha_1M_0\|w_{1,d}-z\|_{L^2(\mathcal{O}_{1,d}\times(0,T))}\|\hat{v}_1\|_{\mathcal{H}_1}+\alpha_2M_0\|w_{2,d}-z\|_{L^2(\mathcal{O}_{2,d}\times(0,T))}\|\hat{v}_2\|_{\mathcal{H}_2}\\
\leq&\left(\alpha_1M_0\|w_{1,d}-z\|_{L^2(\mathcal{O}_{1,d}\times(0,T))}+\alpha_2M_0\|w_{2,d}-z\|_{L^2(\mathcal{O}_{2,d}\times(0,T))}\right)\|(\hat{v}_1,\hat{v}_2)\|_{\mathcal{H}}\\
\leq&\left(\alpha_1M_0\|w_{1,d}\|_{L^2(\mathcal{O}_{1,d}\times(0,T))}+(\alpha_1+\alpha_2)M_0\|z\|_{L^2(Q)})+\alpha_2M_0\|w_{2,d}\|_{L^2(\mathcal{O}_{2,d}\times(0,T))}\right)\|(\hat{v}_1,\hat{v}_2)\|_{\mathcal{H}}.
\end{align*}
Employing energy methods, we conclude that there exists a positive constant $C,$ such that the solution of problem \eqref{eq2.5}satisfies the following inequality:
\begin{align*}
\|z\|_{L^2(Q)}\leq e^{C(1+\|a\|_{L^{\infty}(Q)}+\|B\|_{L^{\infty}(Q)}^2)T}(\|w_0\|_{L^2(\Omega)}+\|f\|_{L^2(\mathcal{O}\times(0,T))}).
\end{align*}
Thus, we infer from the Lax-Milgram Theorem that there exists a unique $(v_1,v_2)\in\mathcal{H},$ such that
\begin{align*}
\langle A(v_1,v_2),(\hat{v}_1,\hat{v}_2)\rangle=\langle B,(\hat{v}_1,\hat{v}_2)\rangle
\end{align*}
for any $(\hat{v}_1,\hat{v}_2)\in\mathcal{H}.$ Moreover, there exists a generic constant $C>0$ depending on $\|a\|_{L^{\infty}(Q)},$ $\|B\|_{L^{\infty}(Q)},$ $\alpha_i,$ $M_0$ and $T,$ such that
\begin{align*}
\|(v_1,v_2)\|_{\mathcal{H}}\leq C\left(\|w_0\|_{L^2(\Omega)}+\|w_{1,d}\|_{L^2(\mathcal{O}_{1,d}\times(0,T))}+\|w_{2,d}\|_{L^2(\mathcal{O}_{2,d}\times(0,T))}+\|f\|_{L^2(\mathcal{O}\times (0,T))}\right).
\end{align*}
\end{proof}
\subsubsection{Characterization of the Nash equilibrium}
In this subsection, for any fixed $f\in L^2(\mathcal{O}\times (0,T)),$ we will express the follower $(v_1(f),v_2(f))$ in terms of a new adjoint variable. To this purpose, let $f\in L^2(\mathcal{O}\times (0,T))$ be given, for any $(v_1,v_2)\in\mathcal{H},$ denote by $w$ the solution of problem \eqref{eq2.1}. In order to simplified equality \eqref{eq2.3}, it is natural to introduce the following adjoint system:
\begin{equation}\label{eq2.10}
\begin{cases}
-\frac{\partial \phi_i}{\partial t}+\Delta^2\phi_i+a(x,t)\phi_i-\nabla\cdot(B(x,t)\phi_i)=\alpha_i(w-w_{i,d})\chi_{\mathcal{O}_{i,d}},\,\,\,\,(x,t)\in Q,\\ 
\phi_i=\frac{\partial\phi_i}{\partial\vec{n}}=0,\,\,\,\,\,(x,t)\in\Sigma,\\
\phi_i(x,T)=0,\,\,\,\,\,x\in\Omega.
\end{cases}
\end{equation}
Thus, combining problem \eqref{eq2.4} with the adjoint system \eqref{eq2.10}, we conclude that $(v_1,v_2)$ is a Nash equilibrium if and only if
\begin{align}
\int_0^T\int_{\mathcal{O}_i}(\phi_i+\mu_iv_i)\hat{v}_i\,dxdt=0
\end{align}
for any $\hat{v}_i\in\mathcal{H}_i,$ which implies that
\begin{align}
v_i=-\frac{\phi_i}{\mu_i}|_{\mathcal{O}_i\times (0,T)}
\end{align}
for any $i=1,2.$ Therefore, we obtain
\begin{equation}\label{eq2.11}
\begin{cases}
\frac{\partial w}{\partial t}+\Delta^2w+a(x,t)w+B(x,t)\cdot\nabla w=f\chi_{\mathcal{O}}-\sum_{i=1}^2\frac{\phi_i}{\mu_i}\chi_{\mathcal{O}_i},\,\,\,\,(x,t)\in Q,\\
-\frac{\partial \phi_i}{\partial t}+\Delta^2\phi_i+a(x,t)\phi_i-\nabla\cdot(B(x,t)\phi_i)=\alpha_i(w-w_{i,d})\chi_{\mathcal{O}_{i,d}},\,\,\,\,(x,t)\in Q,\\ 
w=\frac{\partial w}{\partial\vec{n}}=0,\,\,\,\phi_i=\frac{\partial\phi_i}{\partial\vec{n}}=0,\,\,(x,t)\in\Sigma,\\
w(x,0)=w_0(x),\,\,\phi_i(x,T)=0,\,\,\,x\in\Omega.
\end{cases}
\end{equation}
In what follows, we will prove the well-posedness of problem \eqref{eq2.11} by Banach fixed points Theorem.
\begin{theorem}\label{th3.1}
Assume that $w_0\in L^2(\Omega)$ and $f\in L^2(\mathcal{O}\times(0,T)).$ If $\frac{\max\{\alpha_1,\alpha_2\}}{\min\{\mu_1,\mu_2\}}$ is sufficiently small, then problem \eqref{eq2.11} admits a unique weak solution $(w,\phi_1,\phi_2)\in X\times X\times X,$ where 
\begin{align*}
X=\{u\in L^2(0,T;H_0^2(\Omega)):u_t\in L^2(0,T;H^{-2}(\Omega))\}.
\end{align*}
Moreover, there exists a generic positive constant $C,$ such that
\begin{align*}
\|w\|_X\leq C\left(\|w_0\|_{L^2(\Omega)}+\|f\|_{L^2(\mathcal{O}\times (0,T))}+\sum_{i=1}^2\|w_{i,d}\|_{L^2(\mathcal{O}_{i,d}\times (0,T))}\right).
\end{align*}
\end{theorem}
\begin{proof}
Let $z\in L^2(Q)$ be given, consider the following problem 
\begin{equation}\label{eq3.1.1}
\begin{cases}
\frac{\partial w}{\partial t}+\Delta^2w+a(x,t)w+B(x,t)\cdot\nabla w=f\chi_{\mathcal{O}}-\sum_{i=1}^2\frac{\phi_i}{\mu_i}\chi_{\mathcal{O}_i},\,\,\,\,(x,t)\in Q,\\
-\frac{\partial \phi_i}{\partial t}+\Delta^2\phi_i+a(x,t)\phi_i-\nabla\cdot(B(x,t)\phi_i)=\alpha_i(z-w_{i,d})\chi_{\mathcal{O}_{i,d}},\,\,\,\,(x,t)\in Q,\\ 
w=\frac{\partial w}{\partial\vec{n}}=0,\,\,\,\phi_i=\frac{\partial\phi_i}{\partial\vec{n}}=0,\,\,(x,t)\in\Sigma,\\
w(x,0)=w_0(x),\,\,\phi_i(x,T)=0,\,\,\,x\in\Omega.
\end{cases}
\end{equation}
From the regularity theory of parabolic equations, we conclude that there exists a unique weak solution $(w^z,\phi^z_1,\phi^z_2)\in X\times X\times X.$  Moreover, there exists a positive constant $C,$ such that
\begin{align}\label{eq3.1.2}
\|\phi_i^z\|_X\leq C\alpha_i\|(z-w_{i,d})\chi_{\mathcal{O}_{i,d}}\|_{L^2(Q)}
\end{align}
and
\begin{align}\label{eq3.1.3}
\|w^z\|_X\leq C(\|w_0\|_{L^2(\Omega)}+\|f\chi_{\mathcal{O}}-\sum_{i=1}^2\frac{\phi_i}{\mu_i}\chi_{\mathcal{O}_i}\|_{L^2(Q)}).
\end{align}
Consequently, it follows from inequalities \eqref{eq3.1.2}-\eqref{eq3.1.3} that there exists a generic positive constant $\mathcal{L}_1\geq 1,$ such that
\begin{align}\label{eq3.1.4}
\|w^z\|_X\leq \mathcal{L}_1\left(\|w_0\|_{L^2(\Omega)}+\|f\|_{L^2(\mathcal{O}\times (0,T))}+\sum_{i=1}^2\frac{\alpha_i}{\mu_i}\|(z-w_{i,d})\chi_{\mathcal{O}_{i,d}}\|_{L^2(Q)}\right).
\end{align}
Define the operator $\Lambda:L^2(Q)\rightarrow L^2(Q)$ by
\begin{align*}
\Lambda(z)=w^z,
\end{align*}
then the mapping $\Lambda$ is well-defined. Moreover, we deduce from inequality \eqref{eq3.1.4} that for any $z_1,$ $z_2\in L^2(Q),$ 
\begin{align*}
\|\Lambda(z_1)-\Lambda(z_2)\|_{L^2(Q)}\leq \mathcal{L}_1\sum_{i=1}^2\frac{\alpha_i}{\mu_i}\|z_1-z_2\|_{L^2(Q)}.
\end{align*}
Thus, if $\frac{\max\{\alpha_1,\alpha_2\}}{\min\{\mu_1,\mu_2\}}$ is sufficiently small such that $\mathcal{L}_1\frac{\max\{\alpha_1,\alpha_2\}}{\min\{\mu_1,\mu_2\}}\leq\frac{1}{4},$ the mapping $\Lambda$ is contractive  and possesses a unique fixed point $w\in L^2(Q).$ Moreover, we have
\begin{align}\label{eq3.1.5}
\|w\|_X\leq 2\mathcal{L}_1\left(\|w_0\|_{L^2(\Omega)}+\|f\|_{L^2(\mathcal{O}\times (0,T))}+\sum_{i=1}^2\|w_{i,d}\|_{L^2(\mathcal{O}_{i,d}\times(0,T))}\right).
\end{align}
\end{proof}
\subsection{Null controllability}
The main aim of this subsection is to establish the null controllability of $w$ at $t=T,$ i.e., we will find a distributed control $f\in L^2(\mathcal{O}\times (0,T)),$ such that the solution of problem \eqref{eq2.11} satisfies $w(x,T)\equiv 0$ for any $x\in\Omega.$ To this purpose, we will establish an observability inequality for the adjoint problem of the linearized system of problem \eqref{eq2.11}:
\begin{equation}\label{eq2.12}
\begin{cases}
-\frac{\partial \psi}{\partial t}+\Delta^2\psi+a(x,t)\psi-\nabla\cdot(B(x,t)\psi)=\sum_{i=1}^2\alpha_i\eta_i\chi_{\mathcal{O}_{i,d}},\,\,\,\,(x,t)\in Q,\\ 
\frac{\partial \eta_i}{\partial t}+\Delta^2\eta_i+a(x,t)\eta_i+B(x,t)\cdot\nabla \eta_i=-\frac{\psi}{\mu_i}\chi_{\mathcal{O}_i},\,\,\,\,(x,t)\in Q,\\
\psi=\frac{\partial \psi}{\partial\vec{n}}=0,\,\,\,\eta_i=\frac{\partial\eta_i}{\partial\vec{n}}=0,\,\,(x,t)\in\Sigma,\\
\psi(x,T)=\psi_0(x),\,\,\eta_i(x,0)=0,\,\,\,x\in\Omega.
\end{cases}
\end{equation}
Thus, we will prove the following results.
\begin{proposition}\label{pro2.2}
Assume that $\mathcal{O}_{i,d}\cap \mathcal{O}\neq \emptyset$ for $i=1,2$ and $\frac{\max\{\alpha_1,\alpha_2\}}{\min\{\mu_1,\mu_2\}}$ is sufficiently small.
\begin{enumerate}[(1)]
\item If
\begin{align}\label{1.8}
\mathcal{O}_{1,d}=\mathcal{O}_{2,d}=:\mathcal{O}_d,\,\,\,\zeta_{1,d}=\zeta_{2,d}.
\end{align}
Then there exist a generic positive constant $C,$ such that for any $\psi_0\in L^2(\Omega),$ the solution $(\psi,\eta_1,\eta_2)$ of problem \eqref{eq2.12} satisfies
\begin{align*}
\|\psi(0)\|_{L^2(\Omega)}^2+\int_0^T\int_{\mathcal{O}_d}\theta(x,t)^2|\alpha_1\eta_1+\alpha_2\eta_2|^2\,dxdt\leq C\int_0^T\int_{\mathcal{O}}|\psi|^2\,dxdt
\end{align*}
for some weight function $\theta=\tilde{\xi}^3e^{s\tilde{\alpha}},$ where $\tilde{\xi},$ $\tilde{\alpha}$ are defined in \eqref{a}.
\item If
\begin{align}\label{1.9}
\mathcal{O}_{1,d}\cap \mathcal{O}\neq \mathcal{O}_{2,d}\cap \mathcal{O}.
\end{align}
 Then there exist a generic positive constant $C,$ such that for any $\psi_0\in L^2(\Omega),$ the solution $(\psi,\eta_1,\eta_2)$ of problem \eqref{eq2.12} satisfies
\begin{align*}
\|\psi(0)\|_{L^2(\Omega)}^2+\sum_{i=1}^2\int_0^T\int_{\mathcal{O}_{i,d}}\theta(x,t)^2|\eta_i|^2\,dxdt\leq C\int_0^T\int_{\mathcal{O}}|\psi|^2\,dxdt
\end{align*}
for some weight function $\theta=\min\{\tilde{\xi}_1^3e^{s\tilde{\alpha}_1},\tilde{\xi}_2^3e^{s\tilde{\alpha}_2}\},$ where $\tilde{\xi}_i,$ $\tilde{\alpha}_i$ are defined in \eqref{b}.
\end{enumerate}
\end{proposition}
\begin{proof}
{\bf Case 1: $\mathcal{O}_{1,d}=\mathcal{O}_{2,d}=\mathcal{O}_d.$} Since $\mathcal{O}_d\cap \mathcal{O}\neq\emptyset,$ there exists a non-empty open set $\omega\subset\subset\mathcal{O}_d\cap \mathcal{O}.$ From Lemma \ref{2.1}, we conclude that  there exists a function $\eta\in\mathcal{C}^4(\overline{\Omega})$ such that
\begin{align*}
\eta(x)>0,\,\,\,\,\forall\,\,\,x\in \Omega;\,\,\eta(x)=0,\,\,\,\,\forall\,\,\,x\in\Gamma;\,\,|\nabla\eta(x)|>0,\,\,\,\,\forall\,\,\,x\in\overline{\Omega\backslash\omega}.
\end{align*}
Let $\ell\in\mathcal{C}^{\infty}([0,T])$ be a function satisfying
\begin{equation*}
\ell(t)=
\begin{cases}
\frac{T}{2},\,\,\,\,\forall\,\,\,\,t\in[0,\frac{T}{2}],\\
\sqrt{t(T-t)},\,\,\,\,\,\forall\,\,\,\,\,t\in[\frac{T}{2},T]
\end{cases}
\end{equation*}
and let us introduce the following weight functions:
\begin{align}\label{a}
&\tilde{\alpha}(x,t)=\frac{e^{\lambda(2\|\eta\|_{L^{\infty}(\Omega)}+\eta(x))}-e^{4\lambda\|\eta\|_{L^{\infty}(\Omega)}}}{\ell(t)},\tilde{\xi}(x,t)=\frac{e^{\lambda(2\|\eta\|_{L^{\infty}(\Omega)}+\eta(x))}}{\ell(t)}.
\end{align}
Denote by $h=\alpha_1\eta_1+\alpha_2\eta_2,$ then we conclude from Lemma \ref{2.4} and Lemma \ref{2.5} that there exists $\hat{\lambda}>0$ such that for an arbitrary $\lambda\geq \hat{\lambda},$ we can choose $s_0=s_0(\lambda)>0$ satisfying: there exists a constant $C=C(\lambda)>0$ independent of $s,$ such that for any $s\geq s_0(\lambda)(\sqrt{T}+T),$ one has
\begin{align*}
&\int_Qe^{2s\alpha}\left(s^6\lambda^8\xi^6|\psi|^2+s^4\lambda^6\xi^4|\nabla \psi|^2+s^3\lambda^4\xi^3|\Delta \psi|^2+s^2\lambda^4\xi^2|\nabla^2\psi|^2+s\lambda^2\xi|\nabla\Delta \psi|^2\right)\,dxdt\\
\leq &C\left(\int_{Q_\omega}s^7\lambda^8\xi^7|\psi|^2e^{2s\alpha}\,dxdt+\int_Q(|h\chi_{\mathcal{O}_d}-a\psi|^2+s^2\lambda^2\xi^2|B\psi|^2)e^{2s\alpha}\,dxdt\right)\\
\leq&C\left(\int_{Q_\omega}s^7\lambda^8\xi^7|\psi|^2e^{2s\alpha}\,dxdt+\int_Q(|h|^2+\|a\|_{L^{\infty}(Q)}^2|\psi|^2+s^2\lambda^2\xi^2\|B\|_{L^{\infty}(Q)}^2|\psi|^2)e^{2s\alpha}\,dxdt\right)
\end{align*}
and
\begin{align*}
&\int_Qe^{2s\alpha}\left(s^6\lambda^8\xi^6|h|^2+s^4\lambda^6\xi^4|\nabla h|^2+s^3\lambda^4\xi^3|\Delta h|^2+s^2\lambda^4\xi^2|\nabla^2h|^2+s\lambda^2\xi|\nabla\Delta h|^2\right)\,dxdt\\
&\leq C\left(\int_{Q_\omega}s^7\lambda^8\xi^7|h|^2e^{2s\alpha}\,dxdt+\int_Q|-\frac{\alpha_1}{\mu_1}\chi_{\mathcal{O}_1}\psi-\frac{\alpha_2}{\mu_2}\chi_{\mathcal{O}_2}\psi-ah-B\cdot\nabla h|^2e^{2s\alpha}\,dxdt\right)\\
&\leq C\left(\int_{Q_\omega}s^7\lambda^8\xi^7|h|^2e^{2s\alpha}\,dxdt+\int_Q(|\psi|^2+\|a\|_{L^{\infty}(Q)}^2|h|^2+\|B\|_{L^{\infty}(Q)}^2|\nabla h|^2)e^{2s\alpha}\,dxdt\right),
\end{align*}
 which implies that there exists $\hat{\lambda}>0$ such that for an arbitrary $\lambda\geq \hat{\lambda}(1+\|a\|_{L^{\infty}(Q)}^{\frac{1}{4}}+\|B\|_{L^{\infty}(Q)}^{\frac{1}{3}}),$ we can choose $s_0=s_0(\lambda)>0$ satisfying: there exists a constant $C=C(\lambda)>0$ independent of $s,$ such that for any $s\geq s_0(\lambda)(\sqrt{T}+T),$ one has
\begin{align}\label{eq2.2.1}
\nonumber I(\psi)=&\int_Qe^{2s\alpha}\left(s^6\lambda^8\xi^6|\psi|^2+s^4\lambda^6\xi^4|\nabla \psi|^2+s^3\lambda^4\xi^3|\Delta \psi|^2+s^2\lambda^4\xi^2|\nabla^2\psi|^2+s\lambda^2\xi|\nabla\Delta \psi|^2\right)\,dxdt\\
\leq &C\left(\int_{Q_\omega}s^7\lambda^8\xi^7|\psi|^2e^{2s\alpha}\,dxdt+\int_Q|h|^2e^{2s\alpha}\,dxdt\right)
\end{align}
and
\begin{align}\label{eq2.2.2}
\nonumber I(h)=&\int_Qe^{2s\alpha}\left(s^6\lambda^8\xi^6|h|^2+s^4\lambda^6\xi^4|\nabla h|^2+s^3\lambda^4\xi^3|\Delta h|^2+s^2\lambda^4\xi^2|\nabla^2h|^2+s\lambda^2\xi|\nabla\Delta h|^2\right)\,dxdt\\
\leq &C\left(\int_{Q_\omega}s^7\lambda^8\xi^7|h|^2e^{2s\alpha}\,dxdt+\int_Q|\psi|^2e^{2s\alpha}\,dxdt\right).
\end{align}
Since $\omega\subset\subset\mathcal{O}_d,$ we obtain
\begin{align*}
h=-\frac{\partial \psi}{\partial t}+\Delta^2\psi+a(x,t)\psi-\nabla\cdot(B(x,t)\psi).
\end{align*}
Let $\omega_1$ be an open set such that $\omega\subset\subset\omega_1\subset\subset\mathcal{O}_d\cap \mathcal{O}$ and $\zeta\in C_c^{\infty}(\omega_1)$ is a cut-off function satisfying
\begin{align*}
\zeta(x)\equiv 1,\,\,\,\,\forall\,\,\,x\in\omega;\,\,\,0\leq \zeta(x)\leq 1,\,\,\,\forall\,\,\,x\in\omega_1.
\end{align*}
From the definition of $h,$ we conclude that for any $\lambda\geq \hat{\lambda}(1+\|a\|_{L^{\infty}(Q)}^{\frac{1}{4}}+\|B\|_{L^{\infty}(Q)}^{\frac{1}{3}})$ and $s\geq s_0(\lambda)(T+\sqrt{T}),$
\begin{align}\label{eq2.2.3}
\nonumber&\int_{Q_\omega}s^7\lambda^8\xi^7|h|^2e^{2s\alpha}\,dxdt\leq\int_{Q_{\omega_1}}\zeta s^7\lambda^8\xi^7h\left(-\frac{\partial \psi}{\partial t}+\Delta^2\psi+a(x,t)\psi-\nabla\cdot(B(x,t)\psi)\right)e^{2s\alpha}\,dxdt\\
\nonumber\leq&C\int_{Q_{\omega_1}}\|B\|_{L^{\infty}(Q)}^2s^8\lambda^9\xi^8|h||\psi|e^{2s\alpha}+s^9\lambda^{10}\xi^9|\Delta h||\psi|e^{2s\alpha}\,dxdt\\
\nonumber&+C\int_{Q_{\omega_1}}s^8\lambda^9\xi^8|\nabla\Delta h||\psi|e^{2s\alpha}+s^{11}\lambda^{12}\xi^{11}|h||\psi|e^{2s\alpha}+s^{10}\lambda^{11}\xi^{10}|\nabla h||\psi|e^{2s\alpha}\,dxdt\\
\nonumber&+C\int_{Q_{\omega_1}}s^9\lambda^{10}\xi^9|\nabla^2h||\psi|e^{2s\alpha}+s^8\lambda^8\xi^{10}|\psi||h|e^{2s\alpha}\,dxdt\\
\leq&\epsilon I(h)+C_\epsilon\int_{Q_{\omega_1}}s^{16}\lambda^{16}\xi^{16}|\psi|^2e^{2s\alpha}\,dxdt
\end{align} 
for any sufficiently small $\epsilon>0.$

It follows from inequalities \eqref{eq2.2.1}-\eqref{eq2.2.3} that
\begin{align}\label{eq2.2.4}
I(\psi)+I(h)\leq C\int_{Q_{\omega_1}}s^{16}\lambda^{16}\xi^{16}|\psi|^2e^{2s\alpha}\,dxdt.
\end{align}
Taking the inner product of the first equation of problem \eqref{eq2.12} with $\psi,$ we obtain
\begin{align*}
-\frac{1}{2}\frac{d}{dt}\|\psi(t)\|_{L^2(\Omega)}^2+\|\Delta\psi(t)\|_{L^2(\Omega)}^2\\
=\int_\Omega (h\chi_{\mathcal{O}_d}-a\psi+\nabla\cdot(B\psi))\psi\,dx.
\end{align*}
For any $t\in[0,\frac{T}{2}]$ and $r\in[\frac{T}{4},\frac{3T}{4}],$ applying H\"{o}lder's inequality to yields
\begin{align*}
&\|\psi(t)\|_{L^2(\Omega)}^2+2\int_t^r\|\Delta\psi(\tau)\|_{L^2(\Omega)}^2\,d\tau\\
=&\|\psi(r)\|_{L^2(\Omega)}^2+2\int_t^r\int_\Omega (h\chi_{\mathcal{O}_d}-a\psi+\nabla\cdot(B\psi))\psi\,dxd\tau\\
\leq&\|\psi(r)\|_{L^2(\Omega)}^2+\int_t^r\int_\Omega |h|^2+(1+2\|a\|_{L^{\infty}(Q)}+\|B
\|_{L^{\infty}(Q)}^2)|\psi|^2\,dxd\tau+\frac{1}{4}\int_t^r\|\Delta\psi(\tau)\|_{L^2(\Omega)}^2\,d\tau.
\end{align*}
In particular, it follows from the classical Gronwall inequality that for any $t\in[0,\frac{T}{2}]$ and any $r\in[\frac{T}{4},\frac{3T}{4}],$
\begin{align*}
\|\psi(t)\|_{L^2(\Omega)}^2\leq\left(\|\psi(t+\frac{T}{4})\|_{L^2(\Omega)}^2+\int_t^{t+\frac{T}{4}}\int_\Omega |h(x,r)|^2\,dxdr \right)e^{(1+2\|a\|_{L^{\infty}(Q)}+\|B\|_{L^{\infty}(Q)}^2)\frac{T}{4}}
\end{align*}
and
\begin{align*}
\|\psi(0)\|_{L^2(\Omega)}^2\leq\left(\|\psi(r)\|_{L^2(\Omega)}^2+\int_0^r\int_\Omega |h(x,\tau)|^2\,dxd\tau \right)e^{(1+2\|a\|_{L^{\infty}(Q)}+\|B\|_{L^{\infty}(Q)}^2)\frac{3T}{4}}.
\end{align*}
Denote by $\beta=1+2\|a\|_{L^{\infty}(Q)}+\|B\|_{L^{\infty}(Q)}^2,$ integrating the above inequalities over $[0,\frac{T}{2}]$ and $[\frac{T}{4},\frac{3T}{4}],$ respectively, we obtain
\begin{align*}
\int_0^{\frac{T}{2}}\|\psi(t)\|_{L^2(\Omega)}^2\,dt\leq&\left(\int_{\frac{T}{4}}^{\frac{3T}{4}}\|\psi(t)\|_{L^2(\Omega)}^2\,dt+\int_0^{\frac{T}{2}}\int_t^{t+\frac{T}{4}}\int_\Omega |h(x,r)|^2\,dxdrds \right)e^{\frac{\beta T}{4}}\\
\leq&\left(\int_{\frac{T}{4}}^{\frac{3T}{4}}\|\psi(t)\|_{L^2(\Omega)}^2\,dt+\frac{T}{2}\int_0^{\frac{3T}{4}}\int_\Omega |h(x,t)|^2\,dxdt \right)e^{\frac{\beta T}{4}}
\end{align*}
and
\begin{align*}
\|\psi(0)\|_{L^2(\Omega)}^2\leq&\frac{2}{T}\left(\int_{\frac{T}{4}}^{\frac{3T}{4}}\|\psi(t)\|_{L^2(\Omega)}^2\,dt+\int_{\frac{T}{4}}^{\frac{3T}{4}}\int_0^r\int_\Omega |h(x,\tau)|^2\,dxd\tau \right)e^{\frac{3\beta T}{4}}\\
\leq&\frac{2}{T}\left(\int_{\frac{T}{4}}^{\frac{3T}{4}}\|\psi(t)\|_{L^2(\Omega)}^2\,dt+\frac{T}{2}\int_0^{\frac{3T}{4}}\int_\Omega |h(x,t)|^2\,dxdt \right)e^{\frac{3\beta T}{4}}.
\end{align*}
Thus, we obtain
\begin{align*}
&\|\psi(0)\|_{L^2(\Omega)}^2+\int_0^{\frac{T}{2}}\|\psi(t)\|_{L^2(\Omega)}^2\,dt\\
\leq&(1+\frac{2}{T})\left(\int_{\frac{T}{4}}^{\frac{3T}{4}}\|\psi(t)\|_{L^2(\Omega)}^2\,dt+\int_0^{\frac{3T}{4}}\int_\Omega |h(x,t)|^2\,dxdt \right)e^{\frac{3\beta T}{4}},
\end{align*}
which implies that there exists a positive constant $\mathcal{K}_1,$ such that for any $\lambda\geq \hat{\lambda}(1+\|a\|_{L^{\infty}(Q)}^{\frac{1}{4}}+\|B\|_{L^{\infty}(Q)}^{\frac{1}{3}})$ and $s\geq s_0(\lambda)(T+\sqrt{T}),$
\begin{align*}
&\|\psi(0)\|_{L^2(\Omega)}^2+\int_0^{\frac{T}{2}}\int_\Omega\tilde{\xi}^6e^{2s\tilde{\alpha}}|\psi(x,t)|^2\,dxdt\\
\leq&\mathcal{K}_1\left(\int_{\frac{T}{4}}^{\frac{3T}{4}}\int_\Omega\xi^6e^{2s\alpha}|\psi(x,t)|^2\,dxdt+\int_{\frac{T}{2}}^{\frac{3T}{4}}\int_\Omega\xi^6e^{2s\alpha}|h(x,t)|^2\,dxdt+\int_0^{\frac{T}{2}}\|h(t)\|_{L^2(\Omega)}^2\,dt\right)\\
\leq&\mathcal{K}_1\left(\int_Q\xi^6e^{2s\alpha}|\psi(x,t)|^2\,dxdt+\int_Q\xi^6e^{2s\alpha}|h(x,t)|^2\,dxdt+\int_0^{\frac{T}{2}}\|h(t)\|_{L^2(\Omega)}^2\,dt \right).
\end{align*}
Similarly, we obtain
\begin{align*}
\int_0^{\frac{T}{2}}\|h(t)\|_{L^2(\Omega)}^2\,dt\leq(\frac{\alpha_1}{\mu_1}+\frac{\alpha_2}{\mu_2})^2e^{(\beta+1)\frac{T}{2}}\int_0^{\frac{T}{2}}\|\psi(t)\|_{L^2(\Omega)}^2\,dt.
\end{align*}
Therefore, we conclude that there exists a positive constant $\mathcal{K}_2,$ such that
\begin{align*}
&\|\psi(0)\|_{L^2(\Omega)}^2+\int_0^{\frac{T}{2}}\int_\Omega\tilde{\xi}^6e^{2s\tilde{\alpha}}|\psi(x,t)|^2\,dxdt+\int_0^{\frac{T}{2}}\int_\Omega\tilde{\xi}^6e^{2s\tilde{\alpha}}|h(x,t)|^2\,dxdt\\
\leq&\mathcal{K}_2\left(\int_Q\xi^6e^{2s\alpha}|\psi(x,t)|^2\,dxdt+\int_Q\xi^6e^{2s\alpha}|h(x,t)|^2\,dxdt+(\frac{\alpha_1}{\mu_1}+\frac{\alpha_2}{\mu_2})^2\int_0^{\frac{T}{2}}\int_\Omega\tilde{\xi}^6e^{2s\tilde{\alpha}}|\psi(x,t)|^2\,dxdt \right),
\end{align*}
which entails that if $\frac{\alpha_1}{\mu_1}+\frac{\alpha_2}{\mu_2}$ is sufficiently small such that 
$(\frac{\alpha_1}{\mu_1}+\frac{\alpha_2}{\mu_2})^2\mathcal{K}_2\leq\frac{1}{2},$ we obtain
\begin{align*}
&\|\psi(0)\|_{L^2(\Omega)}^2+\int_0^{\frac{T}{2}}\int_\Omega\tilde{\xi}^6e^{2s\tilde{\alpha}}|\psi(x,t)|^2\,dxdt+\int_0^{\frac{T}{2}}\int_\Omega\tilde{\xi}^6e^{2s\tilde{\alpha}}|h(x,t)|^2\,dxdt\\
\leq&2\mathcal{K}_2\left(\int_Q\xi^6e^{2s\alpha}|\psi(x,t)|^2\,dxdt+\int_Q\xi^6e^{2s\alpha}|h(x,t)|^2\,dxdt\right).
\end{align*}
Since $\tilde{\xi}(x,t)=\xi(x,t)$ and $\tilde{\alpha}(x,t)=\alpha(x,t)$ for any $(x,t)\in\Omega\times (\frac{T}{2},T),$ we deduce from inequality \eqref{eq2.2.4} that there exists a positive constant $\mathcal{K}_3,$ such that for any $\lambda\geq \hat{\lambda}(1+\|a\|_{L^{\infty}(Q)}^{\frac{1}{4}}+\|B\|_{L^{\infty}(Q)}^{\frac{1}{3}})$ and $s\geq s_0(\lambda)(T+\sqrt{T}),$
\begin{align*}
&\int_{\frac{T}{2}}^T\int_\Omega\tilde{\xi}^6e^{2s\tilde{\alpha}}|\psi(x,t)|^2\,dxdt+\int_{\frac{T}{2}}^T\int_\Omega\tilde{\xi}^6e^{2s\tilde{\alpha}}|h(x,t)|^2\,dxdt\\
\leq&\mathcal{K}_3\int_{Q_{\omega_1}}\tilde{\xi}^{16}|\psi|^2e^{2s\tilde{\alpha}}\,dxdt.
\end{align*}
Collecting the above last two inequalities, we conclude that there exists a positive constant $\mathcal{K}_4,$ such that for any $\lambda\geq \hat{\lambda}(1+\|a\|_{L^{\infty}(Q)}^{\frac{1}{4}}+\|B\|_{L^{\infty}(Q)}^{\frac{1}{3}})$ and $s\geq s_0(\lambda)(T+\sqrt{T}),$
\begin{align}\label{eq2.2.20}
\nonumber&\|\psi(0)\|_{L^2(\Omega)}^2+\int_0^T\int_\Omega\tilde{\xi}^6e^{2s\tilde{\alpha}}|\psi(x,t)|^2\,dxdt+\int_0^T\int_\Omega\tilde{\xi}^6e^{2s\tilde{\alpha}}|h(x,t)|^2\,dxdt\\
\nonumber\leq&\mathcal{K}_4\int_{Q_{\omega_1}}\tilde{\xi}^{16}|\psi|^2e^{2s\tilde{\alpha}}\,dxdt\\
\leq&\mathcal{K}_4\int_0^T\int_{\mathcal{O}}\tilde{\xi}^{16}|\psi|^2e^{2s\tilde{\alpha}}\,dxdt.
\end{align}

{\bf Case 2: $\mathcal{O}_{1,d}\cap\mathcal{O}\neq\mathcal{O}_{2,d}\cap\mathcal{O}.$}

Let $\tilde{\mathcal{O}}\subset\subset\mathcal{O}$ be a nonempty open set such that $\mathcal{O}_{i,d}\cap\tilde{\mathcal{O}}\neq\emptyset$ for $i=1,2.$ Thus, there exist nonempty open sets $\omega_i\subset\subset \mathcal{O}_{i,d}\cap\tilde{\mathcal{O}}$ $(i=1,2)$ with $\omega_1\cap \omega_2=\emptyset.$ Moreover, one of the following conditions is satisfied:
\begin{align}\label{eq2.2.5}
\omega_1\cap\mathcal{O}_{2,d}=\omega_2\cap\mathcal{O}_{1,d}=\emptyset
\end{align}
or
\begin{align}\label{eq2.2.6}
\omega_i\subset\mathcal{O}_{j,d},\,\,\,\omega_j\cap\mathcal{O}_{i,d}=\emptyset\,\,\,\,\textit{with}\,\,\,(i,j)=(1,2)\,\,\,\textit{or}\,\,\,(i,j)=(2,1).
\end{align}
From Lemma \ref{2.1}, we conclude that  there exist two functions $\eta^{(1)},$ $\eta^{(2)}\in\mathcal{C}^4(\overline{\Omega})$ such that
\begin{align*}
&\eta^{(i)}(x)>0,\,\,\,\,\forall\,\,\,x\in \Omega;\,\,\eta^{(i)}(x)=0,\,\,\,\,\forall\,\,\,x\in\Gamma;\,\,|\nabla\eta^{(i)}(x)|>0,\,\,\,\,\forall\,\,\,x\in\overline{\Omega\backslash\omega_i};\\
&\eta^{(1)}(x)=\eta^{(2)}(x),\,\,\,\,\forall\,\,\,x\in \Omega\backslash\tilde{\mathcal{O}};\,\,\|\eta^{(1)}\|_{\mathcal{C}(\bar{\Omega})}=\|\eta^{(2)}\|_{\mathcal{C}(\bar{\Omega})}.
\end{align*}
Let us introduce the following weight functions:
\begin{align*}
\alpha_i(x,t)=\frac{e^{\lambda(2\|\eta^{(i)}\|_{\mathcal{C}(\bar{\Omega})}+\eta^{(i)}(x))}-e^{4\lambda\|\eta^{(i)}\|_{\mathcal{C}(\bar{\Omega})}}}{\sqrt{t(T-t)}},\,\,\,\,\,\xi_i(x,t)=\frac{e^{\lambda(2\|\eta^{(i)}\|_{\mathcal{C}(\bar{\Omega})}+\eta^{(i)}(x))}}{\sqrt{t(T-t)}}
\end{align*}
and
\begin{align}\label{b}
\tilde{\alpha}_i(x,t)=\frac{e^{\lambda(2\|\eta^{(i)}\|_{\mathcal{C}(\bar{\Omega})}+\eta^{(i)}(x))}-e^{4\lambda\|\eta^{(i)}\|_{\mathcal{C}(\bar{\Omega})}}}{\ell(t)},\,\,\,\,\,\tilde{\xi}_i(x,t)=\frac{e^{\lambda(2\|\eta^{(i)}\|_{\mathcal{C}(\bar{\Omega})}+\eta^{(i)}(x))}}{\ell(t)}.
\end{align}
From Lemma \ref{2.4}, we deduce that there exists $\hat{\lambda}>0$ such that for an arbitrary $\lambda\geq \hat{\lambda},$ we can choose $s_0=s_0(\lambda)>0$ satisfying: there exists a constant $C=C(\lambda)>0$ independent of $s,$ such that for any $s\geq s_0(\lambda)(\sqrt{T}+T),$ one has
\begin{align*}
&\int_Qe^{2s\alpha_i}\left(s^6\lambda^8\xi_i^6|\eta_i|^2+s^4\lambda^6\xi_i^4|\nabla \eta_i|^2+s^3\lambda^4\xi_i^3|\Delta \eta_i|^2+s^2\lambda^4\xi_i^2|\nabla^2\eta_i|^2+s\lambda^2\xi_i|\nabla\Delta \eta_i|^2\right)\,dxdt\\
&\leq C\left(\int_{Q_{\omega_i}}s^7\lambda^8\xi_i^7|\eta_i|^2e^{2s\alpha_i}\,dxdt+\int_Q|-\frac{1}{\mu_i}\chi_{\mathcal{O}_i}\psi-a\eta_i-B\cdot\nabla \eta_i|^2e^{2s\alpha_i}\,dxdt\right)\\
&\leq C\left(\int_{Q_{\omega_i}}s^7\lambda^8\xi_i^7|\eta_i|^2e^{2s\alpha_i}\,dxdt+\int_Q(|\psi|^2+\|a\|_{L^{\infty}(Q)}^2|\eta_i|^2+\|B\|_{L^{\infty}(Q)}^2|\nabla \eta_i|^2)e^{2s\alpha_i}\,dxdt\right),
\end{align*}
 which implies that there exists $\hat{\lambda}>0$ such that for an arbitrary $\lambda\geq \hat{\lambda}(1+\|a\|_{L^{\infty}(Q)}^{\frac{1}{4}}+\|B\|_{L^{\infty}(Q)}^{\frac{1}{3}}),$ we can choose $s_0=s_0(\lambda)>0$ satisfying: there exists a constant $C=C(\lambda)>0$ independent of $s,$ such that for any $s\geq s_0(\lambda)(\sqrt{T}+T),$ one has
\begin{align}\label{eq2.2.7}
\nonumber I^{(i)}(\eta_i)=&\int_Qe^{2s\alpha_i}\left(s^6\lambda^8\xi_i^6|\eta_i|^2+s^4\lambda^6\xi_i^4|\nabla \eta_i|^2+s^3\lambda^4\xi_i^3|\Delta \eta_i|^2+s^2\lambda^4\xi_i^2|\nabla^2\eta_i|^2+s\lambda^2\xi_i|\nabla\Delta \eta_i|^2\right)\,dxdt\\
\leq &C\left(\int_{Q_{\omega_i}}s^7\lambda^8\xi_i^7|\eta_i|^2e^{2s\alpha_i}\,dxdt+\int_Q|\psi|^2e^{2s\alpha_i}\,dxdt\right).
\end{align}
Let $\hat{\zeta}\in\mathcal{C}^{\infty}(\overline{\Omega})$ be a cut-off function satisfying
\begin{align*}
0\leq\hat{\zeta}(x)\leq 1,\,\,\,\forall\,\,\,x\in\overline{\Omega};\,\,\,\hat{\zeta}(x)\equiv 0,\,\,\,\forall\,\,\,x\in\tilde{\mathcal{O}};\,\,\,\hat{\zeta}(x)\equiv 1,\,\,\,\forall\,\,\,x\in\Omega\backslash\overline{\mathcal{O}}.
\end{align*}
Denote by $\hat{\psi}=\hat{\zeta}\psi,$ then it is clear that $\hat{\psi}$ is the solution of problem
\begin{equation}\label{eq2.2.8}
\begin{cases}
-\frac{\partial \hat{\psi}}{\partial t}+\Delta^2\hat{\psi}+a(x,t)\hat{\psi}-\nabla\cdot(B(x,t)\hat{\psi})=\sum_{i=1}^2\alpha_i\hat{\zeta}\eta_i\chi_{\mathcal{O}_{i,d}}-B\cdot\nabla\hat{\zeta}\psi+2\Delta\hat{\zeta}\Delta\psi\\
+4\nabla\hat{\zeta}\cdot\nabla\Delta\psi+4\nabla\Delta\hat{\zeta}\cdot\nabla\psi+\Delta^2\hat{\zeta}\psi+4\nabla^2\hat{\zeta}:\nabla^2\psi,\,\,\,\,(x,t)\in Q,\\ 
\hat{\psi}=\frac{\partial \hat{\psi}}{\partial\vec{n}}=0,\,\,(x,t)\in\Sigma,\\
\hat{\psi}(x,0)=\zeta\psi_0(x),\,\,\,x\in\Omega.
\end{cases}
\end{equation}
From Lemma \ref{2.5}, we deduce that there exists $\hat{\lambda}>0$ such that for an arbitrary $\lambda\geq \hat{\lambda},$ we can choose $s_0=s_0(\lambda)>0$ satisfying: there exists a constant $C=C(\lambda)>0$ independent of $s,$ such that for any $s\geq s_0(\lambda)(\sqrt{T}+T),$ one has
\begin{align*}
&\int_Qe^{2s\alpha_1}\left(s^6\lambda^8\xi_1^6|\hat{\psi}|^2+s^4\lambda^6\xi_1^4|\nabla \hat{\psi}|^2+s^3\lambda^4\xi_1^3|\Delta \hat{\psi}|^2+s^2\lambda^4\xi_1^2|\nabla^2\hat{\psi}|^2+s\lambda^2\xi_1|\nabla\Delta \hat{\psi}|^2\right)\,dxdt\\
\leq &C\left(\int_{Q_{\omega_1}}s^7\lambda^8\xi_1^7|\hat{\psi}|^2e^{2s\alpha_1}\,dxdt+\int_Q(|\hat{\zeta}\alpha_1\eta_1\chi_{\mathcal{O}_{1,d}}+\hat{\zeta}\alpha_2\eta_2\chi_{\mathcal{O}_{2,d}}-a\hat{\psi}|^2+s^2\lambda^2\xi_1^2|B\hat{\psi}|^2)e^{2s\alpha_1}\,dxdt\right)\\
&+C\int_Q\left(1+\|B\|_{L^{\infty}(Q)}^2+s^2\lambda^2\xi_1^2+s^4\lambda^4\xi_1^4+s^6\lambda^6\xi_1^6\right)|\psi|^2e^{2s\alpha_1}\,dxdt\\
\leq &C\left(\int_{Q_{\omega_1}}s^7\lambda^8\xi_1^7|\hat{\psi}|^2e^{2s\alpha_1}\,dxdt+\int_Q(|\eta_1|^2+|\hat{\zeta}\eta_2|^2+\|a\|_{L^{\infty}(Q)}^2|\hat{\psi}|^2+s^2\lambda^2\xi_1^2\|B\|_{L^{\infty}(Q)}^2|\hat{\psi}|^2)e^{2s\alpha_1}\,dxdt\right)\\
&+C\int_Q\left(1+\|B\|_{L^{\infty}(Q)}^2+s^2\lambda^2\xi_1^2+s^4\lambda^4\xi_1^4+s^6\lambda^6\xi_1^6\right)|\psi|^2e^{2s\alpha_1}\,dxdt,
\end{align*}
which implies that there exists $\hat{\lambda}>0$ such that for an arbitrary $\lambda\geq \hat{\lambda}(1+\|a\|_{L^{\infty}(Q)}^{\frac{1}{4}}+\|B\|_{L^{\infty}(Q)}^{\frac{1}{3}}),$ we can choose $s_0=s_0(\lambda)>0$ satisfying: there exists a constant $C=C(\lambda)>0$ independent of $s,$ such that for any $s\geq s_0(\lambda)(\sqrt{T}+T),$ one has
\begin{align*}
&\int_Qe^{2s\alpha_1}\left(s^6\lambda^8\xi_1^6|\hat{\psi}|^2+s^4\lambda^6\xi_1^4|\nabla \hat{\psi}|^2+s^3\lambda^4\xi_1^3|\Delta \hat{\psi}|^2+s^2\lambda^4\xi_1^2|\nabla^2\hat{\psi}|^2+s\lambda^2\xi_1|\nabla\Delta \hat{\psi}|^2\right)\,dxdt\\
\leq &C\left(\int_{Q_{\omega_1}}s^7\lambda^8\xi_1^7|\hat{\psi}|^2e^{2s\alpha_1}\,dxdt+\int_Q(|\eta_1|^2+|\hat{\zeta}\eta_2|^2)e^{2s\alpha_1}\,dxdt\right)+C\int_Qs^6\lambda^6\xi_1^6|\psi|^2e^{2s\alpha_1}\,dxdt.
\end{align*}
It follows from the definition of $\hat{\zeta}$ that for any $\lambda\geq \hat{\lambda}(1+\|a\|_{L^{\infty}(Q)}^{\frac{1}{4}}+\|B\|_{L^{\infty}(Q)}^{\frac{1}{3}})$ and any $s\geq s_0(\lambda)(\sqrt{T}+T),$
\begin{align}\label{eq2.2.9}
\nonumber&\int_Qe^{2s\alpha_1}\left(s^6\lambda^8\xi_1^6|\psi|^2+s^4\lambda^6\xi_1^4|\nabla \hat{\psi}|^2+s^3\lambda^4\xi_1^3|\Delta \hat{\psi}|^2+s^2\lambda^4\xi_1^2|\nabla^2\hat{\psi}|^2+s\lambda^2\xi_1|\nabla\Delta \hat{\psi}|^2\right)\,dxdt\\
\nonumber\leq &C\left(\int_{Q_{\omega_1}}s^7\lambda^8\xi_1^7|\hat{\psi}|^2e^{2s\alpha_1}\,dxdt+\int_Q(|\hat{\zeta}\eta_1|^2+|\hat{\zeta}\eta_2|^2)e^{2s\alpha_1}\,dxdt\right)\\
&+C\int_0^T\int_{\mathcal{O}}s^6\lambda^8\xi_1^6|\psi|^2e^{2s\alpha_1}\,dxdt.
\end{align}
Thus, we conclude from inequality \eqref{eq2.2.7} and inequality \eqref{eq2.2.9} that for any $\lambda\geq \hat{\lambda}(1+\|a\|_{L^{\infty}(Q)}^{\frac{1}{4}}+\|B\|_{L^{\infty}(Q)}^{\frac{1}{3}})$ and any $s\geq s_0(\lambda)(\sqrt{T}+T),$
\begin{align}\label{eq2.2.10}
\nonumber &I^{(1)}(\eta_1)+I^{(2)}(\eta_2)+\int_Qs^6\lambda^8\xi_1^6|\psi|^2e^{2s\alpha_1}\,dxdt\\
\nonumber\leq &C\left(\int_{Q_{\omega_1}}s^7\lambda^8\xi_1^7|\eta_1|^2e^{2s\alpha_1}\,dxdt+\int_{Q_{\omega_2}}s^7\lambda^8\xi_2^7|\eta_2|^2e^{2s\alpha_2}\,dxdt\right)\\
&+C\left(\int_0^T\int_{\mathcal{O}}s^7\lambda^8\xi_1^7|\psi|^2e^{2s\alpha_1}\,dxdt+\int_0^T\int_{\mathcal{O}}s^7\lambda^8\xi_2^7|\psi|^2e^{2s\alpha_2}\,dxdt\right).
\end{align}
\begin{enumerate}
\item Assume that \eqref{eq2.2.5} holds. In order to eliminate the first two terms of the right hand side of inequality \eqref{eq2.2.10}, we introduce two new open sets $\tilde{\omega}_i\subset\subset\Omega$ with $\omega_i\subset\subset\tilde{\omega}_i\subset\subset\mathcal{O}_{i,d}\cap\tilde{\mathcal{O}}$ and $\tilde{\omega}_i\cap\mathcal{O}_{3-i,d}=\emptyset$ as well as the corresponding cut-off function $\delta_i\in\mathcal{C}_c^{\infty}(\tilde{\omega}_i)$ satisfying
\begin{align*}
0\leq\delta_i(x)\leq 1,\,\,\,\forall\,\,x\in\tilde{\omega}_i;\,\,\,\delta_i(x)\equiv 1,\,\,\,\forall\,\,\,x\in\omega_i.
\end{align*}
By carrying out the similar proof of inequality \eqref{eq2.2.3}, we obtain
\begin{align}\label{eq2.2.11}
\int_{Q_{\omega_i}}s^7\lambda^8\xi_i^7|\eta_i|^2e^{2s\alpha_i}\,dxdt
\leq\epsilon I^{(i)}(\eta_i)+C_\epsilon\int_{Q_{\tilde{\omega}_i}}s^{16}\lambda^{16}\xi_i^{16}|\psi|^2e^{2s\alpha_i}\,dxdt
\end{align} 
for any sufficiently small $\epsilon>0.$

We infer from inequalities \eqref{eq2.2.10}-\eqref{eq2.2.11} that for any $\lambda\geq \hat{\lambda}(1+\|a\|_{L^{\infty}(Q)}^{\frac{1}{4}}+\|B\|_{L^{\infty}(Q)}^{\frac{1}{3}})$ and any $s\geq s_0(\lambda)(\sqrt{T}+T),$
\begin{align}\label{eq2.2.12}
\nonumber &I^{(1)}(\eta_1)+I^{(2)}(\eta_2)+\int_Qs^6\lambda^8\xi_1^6|\psi|^2e^{2s\alpha_1}\,dxdt\\
\leq &C\left(\int_0^T\int_{\mathcal{O}}s^{16}\lambda^{16}\xi_1^{16}|\psi|^2e^{2s\alpha_1}\,dxdt+\int_0^T\int_{\mathcal{O}}s^{16}\lambda^{16}\xi_2^{16}|\psi|^2e^{2s\alpha_2}\,dxdt\right).
\end{align}
\item Assume that \eqref{eq2.2.6} holds. Without loss of generality, we can assume that $(i,j)=(1,2).$
At this point, denote by $h=\alpha_1\eta_1+\alpha_2\eta_2,$ it is clear that $h$ satisfies the following equation
\begin{align*}
h=-\frac{\partial \psi}{\partial t}+\Delta^2\psi+a(x,t)\psi-\nabla\cdot(B(x,t)\psi),\,\,\forall\,\,(x,t)\in Q_{\omega_1}.
\end{align*}
We conclude from Lemma \ref{2.5} that there exists $\hat{\lambda}>0$ such that for an arbitrary $\lambda\geq \hat{\lambda}(1+\|a\|_{L^{\infty}(Q)}^{\frac{1}{4}}+\|B\|_{L^{\infty}(Q)}^{\frac{1}{4}}),$ we can choose $s_0=s_0(\lambda)>0$ satisfying: there exists a constant $C=C(\lambda)>0$ independent of $s,$ such that for any $s\geq s_0(\lambda)(\sqrt{T}+T),$ one has
\begin{align}\label{eq2.2.13}
\nonumber I^{(1)}(h)=&\int_Qe^{2s\alpha_1}\left(s^6\lambda^8\xi_1^6|h|^2+s^4\lambda^6\xi_1^4|\nabla h|^2+s^3\lambda^4\xi_1^3|\Delta h|^2+s^2\lambda^4\xi_1^2|\nabla^2h|^2+s\lambda^2\xi_1|\nabla\Delta h|^2\right)\,dxdt\\
\leq &C\left(\int_{Q_{\omega_1}}s^7\lambda^8\xi_1^7|h|^2e^{2s\alpha_1}\,dxdt+\int_Q|\psi|^2e^{2s\alpha_1}\,dxdt\right).
\end{align}
It follows from inequality \eqref{eq2.2.7} and inequality \eqref{eq2.2.9} that
\begin{align}\label{eq2.2.14}
\nonumber I^{(2)}(\eta_2)=&\int_Qe^{2s\alpha_2}\left(s^6\lambda^8\xi_2^6|\eta_2|^2+s^4\lambda^6\xi_2^4|\nabla \eta_2|^2+s^3\lambda^4\xi_2^3|\Delta \eta_2|^2+s^2\lambda^4\xi_2^2|\nabla^2\eta_2|^2+s\lambda^2\xi_2|\nabla\Delta \eta_2|^2\right)\,dxdt\\
\leq &C\left(\int_{Q_{\omega_2}}s^7\lambda^8\xi_2^7|\eta_2|^2e^{2s\alpha_2}\,dxdt+\int_Q|\psi|^2e^{2s\alpha_2}\,dxdt\right)
\end{align}
and
\begin{align}\label{eq2.2.15}
\nonumber\int_Qs^6\lambda^8\xi_1^6|\psi|^2e^{2s\alpha_1}\,dxdt\leq &C\int_Q(|\eta_2|^2e^{2s\alpha_2}+|h|^2e^{2s\alpha_1})\,dxdt\\
&+C\int_0^T\int_{\mathcal{O}}s^7\lambda^8\xi_1^7|\psi|^2e^{2s\alpha_1}\,dxdt.
\end{align}
Therefore, we infer from inequalities \eqref{eq2.2.13}-\eqref{eq2.2.15} that for any $\lambda\geq \hat{\lambda}(1+\|a\|_{L^{\infty}(Q)}^{\frac{1}{4}}+\|B\|_{L^{\infty}(Q)}^{\frac{1}{3}})$ and any $s\geq s_0(\lambda)(\sqrt{T}+T),$
\begin{align}\label{eq2.2.16}
\nonumber&I^{(1)}(h)+I^{(2)}(\eta_2)+\int_Qs^6\lambda^8\xi_1^6|\psi|^2e^{2s\alpha_1}\,dxdt\\
\nonumber\leq&C\left(\int_{Q_{\omega_1}}s^7\lambda^8\xi_1^7|h|^2e^{2s\alpha_1}\,dxdt+\int_{Q_{\omega_2}}s^7\lambda^8\xi_2^7|\eta_2|^2e^{2s\alpha_2}\,dxdt\right)\\
&+C\left(\int_0^T\int_{\mathcal{O}}s^7\lambda^8\xi_1^7|\psi|^2e^{2s\alpha_1}\,dxdt+\int_0^T\int_{\mathcal{O}}s^7\lambda^8\xi_2^7|\psi|^2e^{2s\alpha_2}\,dxdt\right).
\end{align}
Choosing $\tilde{\omega}_1\subset\subset\Omega$ with $\omega_1\subset\subset\tilde{\omega}_1\subset\subset\mathcal{O}_{1,d}\cap\tilde{\mathcal{O}}$ and let  $\delta_2\in\mathcal{C}_c^{\infty}(\tilde{\omega}_1)$ be a cut-off function satisfying
\begin{align*}
0\leq\delta_2(x)\leq 1,\,\,\,\forall\,\,x\in\tilde{\omega}_1;\,\,\,\delta_2(x)\equiv 1,\,\,\,\forall\,\,\,x\in\omega_1.
\end{align*}
By performing the similar proof of inequality \eqref{eq2.2.3}, we obtain
\begin{align}\label{eq2.2.17}
\int_{Q_{\omega_1}}s^7\lambda^8\xi_1^7|h|^2e^{2s\alpha_1}\,dxdt
\leq\epsilon I^{(1)}(h)+C_\epsilon\int_{Q_{\tilde{\omega}_1}}s^{16}\lambda^{16}\xi_1^{16}|\psi|^2e^{2s\alpha_1}\,dxdt
\end{align} 
and
\begin{align}\label{eq2.2.18}
\int_{Q_{\omega_2}}s^7\lambda^8\xi_2^7|\eta_2|^2e^{2s\alpha_2}\,dxdt
\leq\epsilon I^{(2)}(\eta_2)+C_\epsilon\int_{Q_{\tilde{\omega}_2}}s^{16}\lambda^{16}\xi_2^{16}|\psi|^2e^{2s\alpha_2}\,dxdt
\end{align}
for any sufficiently small $\epsilon>0.$

Thus, we obtain
\begin{align}\label{eq2.2.19}
\nonumber&I^{(1)}(h)+I^{(2)}(\eta_2)+\int_Qs^6\lambda^8\xi_1^6|\psi|^2e^{2s\alpha_1}\,dxdt\\
\leq &C\left(\int_0^T\int_{\mathcal{O}}s^{16}\lambda^{16}\xi_1^{16}|\psi|^2e^{2s\alpha_1}\,dxdt+\int_0^T\int_{\mathcal{O}}s^{16}\lambda^{16}\xi_2^{16}|\psi|^2e^{2s\alpha_2}\,dxdt\right).
\end{align}
\end{enumerate}
By carrying out the similar proof of inequality \eqref{eq2.2.20}, we can also obtain
\begin{align*}
\|\psi(0)\|_{L^2(\Omega)}^2+\sum_{i=1}^2\int_Q\tilde{\xi}_i^6|\eta_i|^2e^{2s\tilde{\alpha}_i}\,dxdt\leq C\int_0^T\int_{\mathcal{O}}|\psi|^2\,dxdt.
\end{align*}
\end{proof}
In what follows, we will consider problem \eqref{1.1} with $F\equiv 0$ and prove the following result.
\begin{theorem}\label{th1.1}
Assume that $\mathcal{O}_{i,d}\cap \mathcal{O}\neq \emptyset$ for $i=1,2$ and either \eqref{1.8} or \eqref{1.9} holds, $F\equiv 0,$ the ratio $\frac{\min\{\mu_1,\mu_2\}}{\max\{\alpha_1,\alpha_2\}}$ is sufficiently large. If $\bar{u}$ is the unique solution of problem \eqref{1.5} with the initial state $\bar{u}(0)=\bar{u}_0\in L^2(\Omega)$ satisfying 
\begin{align}\label{1.10}
\sum_{i=1}^2\int_0^T\int_{\mathcal{O}_{i,d}}\theta(t)^{-2}|\bar{u}(x,t)-\zeta_{i,d}(x,t)|^2\,dxdt<+\infty,
\end{align}
where the weight function $\theta(x,t)$ is the same as in Proposition \ref{pro2.2}. Then for any $u_0\in L^2(\Omega),$ there exist a unique control $f\in L^2(\mathcal{O}\times (0,T))$ and an associated Nash equilibrium $(v_1,v_2),$ such that the solution of problem \eqref{1.1} satisfies \eqref{1.7}. Moreover, there exists a generic positive constant $C,$ such that
\begin{align*}
\int_0^T\int_{\mathcal{O}}|f|^2\,dxdt\leq C\left(\|u_0-\bar{u}_0\|_{L^2(\Omega)}^2+\sum_{i=1}^2\int_0^T\int_{\mathcal{O}_{i,d}}\theta^{-2}(x,t)|\bar{u}(x,t)-\zeta_{i,d}(x,t)|^2\,dxdt\right).
\end{align*}
\end{theorem}
\begin{proof}
For any $\epsilon>0,$ define 
\begin{align*}
G_\epsilon(\psi_0)=\frac{1}{2}\int_0^T\int_{\mathcal{O}}|\psi|^2\,dxdt+\int_\Omega w_0\psi(0)\,dx+\epsilon\|\psi_0\|_{L^2(\Omega)}-\sum_{i=1}^2\int_0^T\int_{\mathcal{O}_{i,d}}\alpha_i\eta_iw_{i,d}\,dxdt.
\end{align*}
Then, it is easy to prove that the functional $G_\epsilon:L^2(\Omega)\rightarrow\mathbb{R}$ is continuous and strictly convex.  Moreover, from the observability inequalities established in Proposition \ref{pro2.2}, we conclude that $G_\epsilon$ is coercive. Therefore, for any $\epsilon>0,$ the functional $G_\epsilon$ admits a unique minimum point $\psi_{0\epsilon}\in L^2(\Omega).$

Denote by $(\psi^\epsilon,\eta_1^\epsilon,\eta_2^\epsilon)$ the solution of problem \eqref{eq2.12} with initial data $\psi_{0\epsilon}.$ If $\psi_{0\epsilon}\neq 0,$ then $G_\epsilon$ satisfies the optimality condition
\begin{align}\label{eq2.2.20}
\int_0^T\int_{\mathcal{O}}\psi_\epsilon\psi\,dxdt+\int_\Omega w_0\psi(0)\,dx+\frac{\epsilon}{\|\psi_{0\epsilon}\|_{L^2(\Omega)}}\int_\Omega\psi_{0\epsilon}\psi_0\,dx-\sum_{i=1}^2\int_0^T\int_{\mathcal{O}_{i,d}}\alpha_i\eta_iw_{i,d}\,dxdt=0
\end{align}
for any $\psi_0\in L^2(\Omega),$ where $(\psi,\eta_1,\eta_2)$ solves problem \eqref{eq2.12} with initial data $\psi_0.$

Now, let $f_\epsilon=\psi^\epsilon$ and let $(w^\epsilon,\phi_1^\epsilon, \phi_2^\epsilon)$ be the solution of problem \eqref{eq2.11} with $f=f_\epsilon.$ Combining problem \eqref{eq2.11} and problem \eqref{eq2.12} with initial data $\psi_0,$ we obtain
\begin{align}\label{eq2.2.21}
\int_Q\left(\psi^\epsilon\chi_{\mathcal{O}}-\sum_{i=1}^2\frac{\phi^\epsilon_i}{\mu_i}\chi_{\mathcal{O}_i}\right)\psi\,dxdt=\int_\Omega w^\epsilon(x,T)\psi_0(x)-w_0(x)\psi(x,0)\,dx+\int_Q\sum_{i=1}^2\alpha_i\eta_i\chi_{\mathcal{O}_{i,d}}w^\epsilon\,dxdt
\end{align}
and 
\begin{align}\label{eq2.2.22}
\sum_{i=1}^2\int_Q\alpha_i(w^\epsilon-w_{i,d})\chi_{\mathcal{O}_{i,d}}\eta_i\,dxdt=-\sum_{i=1}^2\int_Q\frac{\psi}{\mu_i}\chi_{\mathcal{O}_i}\phi_i^\epsilon\,dxdt.
\end{align}
Along with inequalities \eqref{eq2.2.20}-\eqref{eq2.2.22}, we obtain
\begin{align}\label{eq2.2.23}
\int_\Omega w^\epsilon(x,T)\psi_0(x)\,dx=-\frac{\epsilon}{\|\psi_{0\epsilon}\|_{L^2(\Omega)}}\int_\Omega\psi_{0\epsilon}\psi_0\,dx
\end{align}
for any $\psi_0\in L^2(\Omega),$ which implies that
\begin{align}\label{eq2.2.24}
\|w^\epsilon(T)\|_{L^2(\Omega)}\leq \epsilon.
\end{align}
If $\psi_{0\epsilon}=0,$ then 
\begin{align*}
\lim_{t\rightarrow 0^+}\frac{G_\epsilon(t\psi_0)}{t}\geq 0
\end{align*}
for any $\psi_0\in L^2(\Omega),$ i.e., 
\begin{align}\label{eq2.2.25}
\int_\Omega w_0\psi(0)\,dx+\epsilon\|\psi_0\|_{L^2(\Omega)}-\sum_{i=1}^2\int_0^T\int_{\mathcal{O}_{i,d}}\alpha_i\eta_iw_{i,d}\,dxdt\geq0,
\end{align}
where $(\psi,\eta_1,\eta_2)$ is the solution of problem \eqref{eq2.12} with initial data $\psi_0\in L^2(\Omega).$ Hence, we conclude from inequalities \eqref{eq2.2.21}-\eqref{eq2.2.22}, \eqref{eq2.2.25} and the fact that $\psi_\epsilon=0$ that
\begin{align*}
\epsilon\|\psi_0\|_{L^2(\Omega)}+\int_\Omega w^\epsilon(x,T)\psi_0\,dx\geq 0
\end{align*}
for any $\psi_0\in L^2(\Omega),$ which also implies that
\begin{align}\label{eq2.2.26}
\|w^\epsilon(T)\|_{L^2(\Omega)}\leq \epsilon.
\end{align}
Therefore, the solution $(w^\epsilon,\phi_1^\epsilon,\phi_2^\epsilon)$ of problem \eqref{eq2.11} associated with $f_\epsilon$ satisfies inequality \eqref{eq2.2.26}. Moreover, we obtain
\begin{align*}
\int_0^T\int_{\mathcal{O}}|f_\epsilon|^2\,dxdt\leq C\left(\|w_0\|_{L^2(\Omega)}^2+\sum_{i=1}^2\int_0^T\int_{\mathcal{O}_{i,d}}\theta^{-2}(x,t)|w_{i,d}(x,t)|^2\,dxdt\right),
\end{align*}
which entails that the controls $\{f_\epsilon\}_{\epsilon>0}$ are uniformly bounded in $L^2(\mathcal{O}\times (0,T)).$ Without loss of generality, we can assume that $f_\epsilon\rightharpoonup f$ weakly in $L^2(\mathcal{O}\times (0,T))$ and 
\begin{align*}
(w^\epsilon,\phi_1^\epsilon,\phi_2^\epsilon)\rightharpoonup (w,\phi_1,\phi_2),\,\,\,\,\,\textit{weakly\,\,\,in}\,\,\,X\times X\times X,
\end{align*}
where $(w,\phi_1,\phi_2)$ is the solution of problem \eqref{eq2.11} with $f.$ In particular, we have the weak convergence of $w^\epsilon(T)$ in $L^2(\Omega).$ Thus, we conclude from inequality \eqref{eq2.2.26} that  $w(T)\equiv 0,$ i.e., $f$ is the desired control. Since $J(f)$ is strictly convex and problem \eqref{1.1} with $F\equiv 0$ is linear, then the desired control $f$ is unique. Moreover, we have
\begin{align*}
\int_0^T\int_{\mathcal{O}}|f|^2\,dxdt\leq C\left(\|w_0\|_{L^2(\Omega)}^2+\sum_{i=1}^2\int_0^T\int_{\mathcal{O}_{i,d}}\theta^{-2}(x,t)|w_{i,d}(x,t)|^2\,dxdt\right).
\end{align*}
\end{proof}
\begin{remark}
The assumption \eqref{1.10} is natural. Indeed, we would like to obtain \eqref{1.7} and simultaneously keep $u$ be not too far from $\zeta_{i,d}$ in $\mathcal{O}_{i,d}\times (0,T).$ Consequently, it is reasonable to impose the assumptions that the functions $\zeta_{i,d}$ approach to $\bar{u}$ in $\mathcal{O}_{i,d}$ as $t$ goes to $T.$
\end{remark}
\section{The semi-linear case}
\def\theequation{4.\arabic{equation}}\makeatother
\setcounter{equation}{0}
Note that the cost functional $J_i$ are convex and continuously differentiable in the linear case. Thus, \eqref{1.4} is equivalent to 
\begin{align}\label{1.11}
J_i'(f;v_1,v_2)\hat{v}_i=0
\end{align}
for any $\hat{v}_i\in \mathcal{H}_i$ and any $i=1,2.$ However, in the more general semi-linear case,  the functionals $J_i$ are not convex in general. Thus, it is necessary to introduce the following definition.
\begin{definition}\label{de1.1}
For any given $f\in L^2(\mathcal{O}\times(0,T)),$ a pair $(v_1,v_2)$ is said to be a Nash quasi-equilibrium for the functional $J_i$ associated to $f,$ if the condition \eqref{1.11} is satisfied.
\end{definition}
In this section, we will always assume $F\in W^{1,\infty}(\mathbb{R}^{n+1},\mathbb{R})$ and establish the exact controllability of problem \eqref{1.1} in the semi-linear case and characterize the relation between Nash quasi-equilibrium and \eqref{1.4}.
\subsection{The optimality system in the semi-linear case}
In this subsection, we will obtain an optimality system that describes any Nash quasi-equilibrium. First of all, from definition \ref{de1.1}, we deduce that a pair $(v_1,v_2)$ is a Nash equilibrium if and only if
\begin{align}\label{eq4.1.1}
\alpha_i\int_0^T\int_{\mathcal{O}_{i,d}}(u-\zeta_{i,d})\hat{u}_i\,dxdt+\mu_i\int_0^T\int_{\mathcal{O}_i}v_i\hat{v}_i\,dxdt=0
\end{align}
for any $\hat{v}_i\in\mathcal{H}_i,$ where $\hat{u}_i$ is the solution of problem 
\begin{equation}\label{eq4.1.2}
\begin{cases}
\frac{\partial \hat{u}_i}{\partial t}+\Delta^2\hat{u}_i+a(x,t)\hat{u}_i+B(x,t)\cdot\nabla \hat{u}_i\\
=F_u(u,\nabla u)\hat{u}_i+\nabla_pF(u,\nabla u)\cdot\nabla\hat{u}_i+\hat{v}_i\chi_{\mathcal{O}_i},\,\,\,\,(x,t)\in Q,\\
\hat{u}_i=\frac{\partial \hat{u}_i}{\partial\vec{n}}=0,\,\,\,\,\,(x,t)\in\Sigma,\\
\hat{u}_i(x,0)=0,\,\,\,\,\,x\in\Omega,
\end{cases}
\end{equation}
where $p=\nabla u.$

In order to further simplified equality \eqref{eq4.1.1}, we also introduce the adjoint system of problem \eqref{eq4.1.2}:
\begin{equation}\label{eq4.1.3}
\begin{cases}
-\frac{\partial \phi_i}{\partial t}+\Delta^2\phi_i+a(x,t)\phi_i-\nabla\cdot(B(x,t)\phi_i)\\
=F_u(u,\nabla u)\phi_i-\nabla\cdot(\nabla_pF(u,\nabla u)\phi_i)+\alpha_i(u-\zeta_{i,d})\chi_{\mathcal{O}_{i,d}},\,\,\,\,(x,t)\in Q,\\ 
\phi_i=\frac{\partial\phi_i}{\partial\vec{n}}=0,\,\,\,\,\,(x,t)\in\Sigma,\\
\phi_i(x,T)=0,\,\,\,\,\,x\in\Omega.
\end{cases}
\end{equation}
Thus, by combining problem \eqref{eq4.1.2} and problem \eqref{eq4.1.3}, we can reformulate equality \eqref{eq4.1.1} as follows
\begin{align*}
\int_0^T\int_{\mathcal{O}_i}(\mu_iv_i+\phi_i)\hat{v}_i\,dxdt=0
\end{align*}
for any $\hat{v}_i\in\mathcal{H}_i,$ which implies that
\begin{align}\label{eq4.1.4}
v_i=-\frac{\phi_i}{\mu_i}|_{\mathcal{O}_i\times (0,T)}.
\end{align}
Consequently, we obtain the following optimality system
\begin{equation}\label{eq4.1.5}
\begin{cases}
\frac{\partial u}{\partial t}+\Delta^2u+a(x,t)u+B(x,t)\cdot\nabla u=F(u,\nabla u)+f\chi_{\mathcal{O}}-\sum_{i=1}^2\frac{\phi_i}{\mu_i}\chi_{\mathcal{O}_i},\,\,\,\,(x,t)\in Q,\\
-\frac{\partial \phi_i}{\partial t}+\Delta^2\phi_i+a(x,t)\phi_i-\nabla\cdot(B(x,t)\phi_i)\\
=F_u(u,\nabla u)\phi_i-\nabla\cdot(\nabla_pF(u,\nabla u)\phi_i)+\alpha_i(u-\zeta_{i,d})\chi_{\mathcal{O}_{i,d}},\,\,\,\,(x,t)\in Q,\\ 
u=\frac{\partial u}{\partial\vec{n}}=0,\,\,\,\phi_i=\frac{\partial\phi_i}{\partial\vec{n}}=0,\,\,\,\,\,\,\,(x,t)\in\Sigma,\\
u(x,0)=u_0(x),\,\,\phi_i(x,T)=0,\,\,\,x\in\Omega.
\end{cases}
\end{equation}
In what follows, we will also establish the existence of solutions of problem \eqref{eq4.1.5} by Leray-Schauder's fixed points Theorem under some suitable assumptions.
\begin{theorem}\label{th3.2}
Assume that $u_0\in L^2(\Omega),$ $f\in L^2(\mathcal{O}\times(0,T))$ and $F\in W^{1,\infty}(\mathbb{R}^{n+1},\mathbb{R}).$ If $\frac{\max\{\alpha_1,\alpha_2\}}{\min\{\mu_1,\mu_2\}}$ is sufficiently small, then problem \eqref{eq4.1.5} admits a unique weak solution $(u,\phi_1,\phi_2)\in X\times X\times X.$ Moreover, there exists a generic positive constant $C,$ such that
\begin{align*}
\|u\|_X\leq C\left(\|u_0\|_{L^2(\Omega)}+1+\|f\|_{L^2(\mathcal{O}\times (0,T))}+\sum_{i=1}^2\|\zeta_{i,d}\|_{L^2(\mathcal{O}_{i,d}\times(0,T))}\right).
\end{align*}
\end{theorem}
\begin{proof}
Let $z\in L^2(0,T;H_0^1(\Omega))$ be given, consider the following problem 
\begin{equation}\label{eq3.1.1}
\begin{cases}
\frac{\partial u}{\partial t}+\Delta^2u+a(x,t)u+B(x,t)\cdot\nabla u=\tilde{G}_1(z,\nabla z)u+\tilde{G}_2(z,\nabla z)\cdot \nabla u\\
+F(0,0)+f\chi_{\mathcal{O}}-\sum_{i=1}^2\frac{\phi_i}{\mu_i}\chi_{\mathcal{O}_i},\,\,\,\,(x,t)\in Q,\\
-\frac{\partial \phi_i}{\partial t}+\Delta^2\phi_i+a(x,t)\phi_i-\nabla\cdot(B(x,t)\phi_i)=F_u(z,\nabla z)\phi_i-\nabla\cdot(\nabla_pF(z,\nabla z)\phi_i)\\
+\alpha_i(z-\zeta_{i,d})\chi_{\mathcal{O}_{i,d}},\,\,\,\,(x,t)\in Q,\\ 
u=\frac{\partial u}{\partial\vec{n}}=0,\,\,\,\phi_i=\frac{\partial\phi_i}{\partial\vec{n}}=0,\,\,(x,t)\in\Sigma,\\
u(x,0)=u_0(x),\,\,\phi_i(x,T)=0,\,\,\,x\in\Omega,
\end{cases}
\end{equation}
where
\begin{align*}
\tilde{G}_1(w,\nabla w)=\int_0^1\frac{\partial F}{\partial u}(\tau w,\tau\nabla w)\,d\tau,\,\,\,\tilde{G}_2(w,\nabla w)=\int_0^1\nabla_pF(\tau w,\tau\nabla w)\,d\tau.
\end{align*}
From the regularity theory of parabolic equations, we conclude that there exists a unique weak solution $(u^z,\phi^z_1,\phi^z_2)\in X\times X\times X.$  Moreover, there exists a positive constant $C,$ such that
\begin{align}\label{eq4.1.6}
\|\phi_i^z\|_X\leq C\alpha_i\|(z-\zeta_{i,d})\chi_{\mathcal{O}_{i,d}}\|_{L^2(Q)}
\end{align}
and
\begin{align}\label{eq4.1.7}
\|u^z\|_X\leq C(\|u_0\|_{L^2(\Omega)}+\|F(0,0)+f\chi_{\mathcal{O}}-\sum_{i=1}^2\frac{\phi_i}{\mu_i}\chi_{\mathcal{O}_i}\|_{L^2(Q)}).
\end{align}
Consequently, it follows from inequalities \eqref{eq4.1.6}-\eqref{eq4.1.7} that there exists a generic positive constant $\mathcal{L}_1,$ such that
\begin{align}\label{eq4.1.8}
\|u^z\|_X\leq \mathcal{L}_1\left(\|u_0\|_{L^2(\Omega)}+1+\|f\|_{L^2(\mathcal{O}\times (0,T))}+\sum_{i=1}^2\frac{\alpha_i}{\mu_i}\|(z-\zeta_{i,d})\chi_{\mathcal{O}_{i,d}}\|_{L^2(Q)}\right).
\end{align}
Define $\Lambda:L^2(0,T;H_0^1(\Omega))\rightarrow L^2(0,T;H_0^1(\Omega))$ by
\begin{align*}
\Lambda(z)=u^z,
\end{align*}
then the mapping $\Lambda$ is well-defined. In what follows, we will prove the existence of solution for problem \eqref{eq4.1.5} by the Leray-Schauder's fixed points Theorem. To this purpose, we will first prove that $\Lambda$ is continuous, i.e., if $z_j\rightarrow z$ in $L^2(0,T;H_0^1(\Omega)),$ we have $\Lambda(z_j)\rightarrow \Lambda(z).$

Denote by $u^j=\Lambda(z_j),$ where $(u^j,\phi_1^j,\phi_2^j)$ is the solution of problem
\begin{equation}\label{eq3.1.5}
\begin{cases}
\frac{\partial u^j}{\partial t}+\Delta^2u^j+a(x,t)u^j+B(x,t)\cdot\nabla u^j=\tilde{G}_1(z_j,\nabla z_j)u^j+\tilde{G}_2(z_j,\nabla z_j)\cdot \nabla u^j\\
+F(0,0)+f\chi_{\mathcal{O}}-\sum_{i=1}^2\frac{\phi^j_i}{\mu_i}\chi_{\mathcal{O}_i},\,\,\,\,(x,t)\in Q,\\
-\frac{\partial \phi^j_i}{\partial t}+\Delta^2\phi^j_i+a(x,t)\phi^j_i-\nabla\cdot(B(x,t)\phi^j_i)=F_u(z_j,\nabla z_j)\phi^j_i-\nabla\cdot(\nabla_pF(z_j,\nabla z_j)\phi^j_i)\\
+\alpha_i(z_j-\zeta_{i,d})\chi_{\mathcal{O}_{i,d}},\,\,\,\,(x,t)\in Q,\\ 
u^j=\frac{\partial u^j}{\partial\vec{n}}=0,\,\,\,\phi^j_i=\frac{\partial\phi^j_i}{\partial\vec{n}}=0,\,\,(x,t)\in\Sigma,\\
u^j(x,0)=u_0(x),\,\,\phi^j_i(x,T)=0,\,\,\,x\in\Omega.
\end{cases}
\end{equation}
It follows from inequalities \eqref{eq4.1.6}, \eqref{eq4.1.8} and the fact that $z_j\rightarrow z$ in $L^2(0,T;H_0^1(\Omega))$ that
\begin{align*}
&\{u^j\}_{j=1}^{\infty}\,\,\,\textit{is\,\,\,uniformly\,\,\,bounded\,\,in}\,\,X,\\
&\{\phi_i^j\}_{j=1}^{\infty}\,\,\,\textit{is\,\,\,uniformly\,\,\,bounded\,\,in}\,\,X,\,\,\textit{for}\,\,i=1,2,
\end{align*}
which entails that there exists a subsequence of $\{u^j\}_{j=1}^\infty,$ $\{\phi_i^j\}_{j=1}^\infty$ (still denote by themselves) and $u\in X,$ $\phi_i\in X,$ such that
\begin{align*}
&u^j\rightharpoonup u\,\,\,\textit{in}\,\,X\,\,\,\textit{as}\,\,j\rightarrow+\infty,\\
&u^j\rightarrow u\,\,\,\textit{in}\,\,L^2(0,T;H_0^1(\Omega))\,\,\,\textit{as}\,\,j\rightarrow+\infty,\\
&\phi_i^j\rightharpoonup\phi_i\,\,\,\textit{in}\,\,X\,\,\,\textit{as}\,\,j\rightarrow+\infty,\,\,\textit{for}\,\,i=1,2.
\end{align*}
Since $F\in W^{1,\infty}(\mathbb{R}^{n+1},\mathbb{R}),$ we conclude that there exists a subsequence of $\{\tilde{G}_1(z_j,\nabla z_j)\}_{j=1}^\infty,$ $\{\tilde{G}_2(z_j,\nabla z_j)\}_{j=1}^\infty,$ $\{F_u(z_j,\nabla z_j)\}_{j=1}^\infty,$
$\{\nabla_pF(z_j,\nabla z_j)\}_{j=1}^\infty$ (still denote by themselves), such that
\begin{align*}
&\tilde{G}_1(z_j,\nabla z_j)\rightarrow\tilde{G}_1(z,\nabla z)\,\,\,\textit{weakly\,\,star\,\,in}\,\,L^{\infty}(Q),\,\,\textit{as}\,\,j\rightarrow+\infty,\\
&\tilde{G}_2(z_j,\nabla z_j)\rightarrow \tilde{G}_2(z,\nabla z)\,\,\,\textit{weakly\,\,star\,\,in}\,\,L^{\infty}(Q),\,\,\textit{as}\,\,j\rightarrow+\infty,\\
&F_u(z_j,\nabla z_j)\rightarrow F_u(z,\nabla z)\,\,\,\textit{weakly\,\,star\,\,in}\,\,L^{\infty}(Q),\,\,\textit{as}\,\,j\rightarrow+\infty,\\
&\nabla_pF(z_j,\nabla z_j)\rightarrow\nabla_pF(z,\nabla z)\,\,\,\textit{weakly\,\,star\,\,in}\,\,L^{\infty}(Q),\,\,\textit{as}\,\,j\rightarrow+\infty.
\end{align*}
Let $j\rightarrow+\infty$ in problem \eqref{eq3.1.5}, we obtain
\begin{equation}\label{eq3.1.6}
\begin{cases}
\frac{\partial u}{\partial t}+\Delta^2u+a(x,t)u+B(x,t)\cdot\nabla u=\tilde{G}_1(z,\nabla z)u+\tilde{G}_2(z,\nabla z)\cdot \nabla u\\
+F(0,0)+f\chi_{\mathcal{O}}-\sum_{i=1}^2\frac{\phi_i}{\mu_i}\chi_{\mathcal{O}_i},\,\,\,\,(x,t)\in Q,\\
-\frac{\partial \phi_i}{\partial t}+\Delta^2\phi_i+a(x,t)\phi_i-\nabla\cdot(B(x,t)\phi_i)=F_u(z,\nabla z)\phi_i-\nabla\cdot(\nabla_pF(z,\nabla z)\phi_i)\\
+\alpha_i(z-\zeta_{i,d})\chi_{\mathcal{O}_{i,d}},\,\,\,\,(x,t)\in Q,\\ 
u=\frac{\partial u}{\partial\vec{n}}=0,\,\,\,\phi_i=\frac{\partial\phi_i}{\partial\vec{n}}=0,\,\,(x,t)\in\Sigma,\\
u(x,0)=u_0(x),\,\,\phi_i(x,T)=0,\,\,\,x\in\Omega,
\end{cases}
\end{equation}
which entails that $u=\Lambda(z).$ Thus, we have proved that $\Lambda(z_j)\rightarrow\Lambda(z)$ in $L^2(0,T;H_0^1(\Omega)),$ i.e., the mapping $\Lambda:L^2(0,T;H_0^1(\Omega))\rightarrow L^2(0,T;H_0^1(\Omega))$ is continuous. Thanks to the compactness of $X\subset L^2(0,T;H_0^1(\Omega))$ and inequality \eqref{eq4.1.8}, we conclude that the mapping $\Lambda:L^2(0,T;H_0^1(\Omega))\rightarrow L^2(0,T;H_0^1(\Omega))$ is compact. 
Denote by 
\begin{align*}
\mathcal{R}_1=2\mathcal{L}_1\left(\|w_0\|_{L^2(\Omega)}+1+\|f\|_{L^2(\mathcal{O}\times (0,T))}+\sum_{i=1}^2\|w_{i,d}\|_{L^2(\mathcal{O}_{i,d}\times (0,T))}\right)
\end{align*}
and
\begin{align*}
B=\{u\in L^2(0,T;H_0^1(\Omega)):\|u\|_{L^2(0,T;H_0^1(\Omega))}\leq\mathcal{R}_1\},
\end{align*}
if $\frac{\max\{\alpha_1,\alpha_2\}}{\min\{\mu_1,\mu_2\}}$ is sufficiently small such that $\mathcal{L}_1\sum_{i=1}^2\frac{\alpha_i}{\mu_i}\leq\frac{1}{2},$ then $\Lambda: B\rightarrow B.$ Thus, we can employ the Leray-Schauder's fixed points Theorem to conclude that the operator $\Lambda$ possesses at least one fixed point $w\in L^2(0,T;H_0^1(\Omega)).$ Moreover, we have
\begin{align}\label{eq3.1.5}
\|w\|_X\leq 2\mathcal{L}_1\left(\|w_0\|_{L^2(\Omega)}+1+\|f\|_{L^2(\mathcal{O}\times (0,T))}+\sum_{i=1}^2\|w_{i,d}\|_{L^2(\mathcal{O}_{i,d}\times (0,T))}\right).
\end{align}
\end{proof}
\subsection{Exact controllability}
In this subsection, we will prove the exact controllability of problem \eqref{1.1} and obtain the following result.
\begin{theorem}\label{th1.2}
Suppose that the subset $\mathcal{O}_{i,d}$ and $\mu_i$ are the same as in Theorem \ref{th1.1}, $F\in W^{1,\infty}(\mathbb{R}^{n+1};\mathbb{R})$ and let $\bar{u}$ be the unique solution of problem \eqref{1.5} with initial data $\bar{u}_0\in L^2(\Omega).$ If \eqref{1.10} holds, then for any $u_0\in L^2(\Omega),$ there exists a control $f\in L^2(\mathcal{O}\times (0,T))$ and an associated Nash quasi-equilibria $(v_1,v_2)$ such that the corresponding solutions to problem \eqref{1.1} satisfy \eqref{1.7}.
\end{theorem}
\begin{proof}
Denote by $w=u-\bar{u},$ then we can rewrite problem \eqref{eq4.1.5} as follows:
\begin{equation}\label{eq4.2.1}
\begin{cases}
\frac{\partial w}{\partial t}+\Delta^2w+a(x,t)w+B(x,t)\cdot\nabla w=G_1(w,\nabla w)w+G_2(w,\nabla w)\cdot \nabla w\\
+f\chi_{\mathcal{O}}-\sum_{i=1}^2\frac{\phi_i}{\mu_i}\chi_{\mathcal{O}_i},\,\,\,\,(x,t)\in Q,\\
-\frac{\partial \phi_i}{\partial t}+\Delta^2\phi_i+a(x,t)\phi_i-\nabla\cdot(B(x,t)\phi_i)=F_u(\bar{u}+w,\nabla \bar{u}+\nabla w)\phi_i\\
-\nabla\cdot(\nabla_pF(\bar{u}+w,\nabla \bar{u}+\nabla w)\phi_i)+\alpha_i(w-w_{i,d})\chi_{\mathcal{O}_{i,d}},\,\,\,\,(x,t)\in Q,\\ 
w=\frac{\partial w}{\partial\vec{n}}=0,\,\,\,\phi_i=\frac{\partial\phi_i}{\partial\vec{n}}=0,\,\,\,\,\,\,\,(x,t)\in\Sigma,\\
w(x,0)=w_0(x),\,\,\phi_i(x,T)=0,\,\,\,x\in\Omega,
\end{cases}
\end{equation}
where $w_{i,d}=\zeta_{i,d}-\bar{u},$ $w_0=u_0-\bar{u}_0$ and 
\begin{align*}
G_1(w,\nabla w)=&\int_0^1\frac{\partial F}{\partial u}(\bar{u}+\tau w,\nabla\bar{u}+\tau\nabla w)\,d\tau,\\
G_2(w,\nabla w)=&\int_0^1\nabla_pF(\bar{u}+\tau w,\nabla\bar{u}+\tau\nabla w)\,d\tau.
\end{align*}
For any given $z\in L^2(0,T;H_0^1(\Omega))$ and any fixed $f\in L^2(\mathcal{O}\times (0,T)),$ we consider the following problem:
\begin{equation}\label{eq4.2.2}
\begin{cases}
\frac{\partial w}{\partial t}+\Delta^2w+a(x,t)w+B(x,t)\cdot\nabla w=G_1(z,\nabla z)w+G_2(z,\nabla z)\cdot \nabla w\\
+f\chi_{\mathcal{O}}-\sum_{i=1}^2\frac{\phi_i}{\mu_i}\chi_{\mathcal{O}_i},\,\,\,\,(x,t)\in Q,\\
-\frac{\partial \phi_i}{\partial t}+\Delta^2\phi_i+a(x,t)\phi_i-\nabla\cdot(B(x,t)\phi_i)=F_u(\bar{u}+z,\nabla \bar{u}+\nabla z)\phi_i\\
-\nabla\cdot(\nabla_pF(\bar{u}+z,\nabla \bar{u}+\nabla z)\phi_i)+\alpha_i(w-w_{i,d})\chi_{\mathcal{O}_{i,d}},\,\,\,\,(x,t)\in Q,\\ 
w=\frac{\partial w}{\partial\vec{n}}=0,\,\,\,\phi_i=\frac{\partial\phi_i}{\partial\vec{n}}=0,\,\,\,\,\,\,\,(x,t)\in\Sigma,\\
w(x,0)=w_0(x),\,\,\phi_i(x,T)=0,\,\,\,x\in\Omega.
\end{cases}
\end{equation}
Since $F\in W^{1,\infty}(\mathbb{R}^{n+1},\mathbb{R}),$ there exists a positive constant $M,$ such that
\begin{align*}
|G_1(z,p)|+|G_2(z,p)|+|F_u(z,p)|+|\nabla_pF(z,p)|\leq M,\,\,\,\,\,\forall\,\,\,\,(z,p)\in\mathbb{R}^{n+1}.
\end{align*}
Arguing as in the proof of Theorem \ref{th3.1}, we conclude that there exists a constant $C>0,$ such that
\begin{align*}
\|w\|_X\leq C\left(1+\|f\|_{L^2(\mathcal{O}\times (0,T))}\right).
\end{align*}
For any given $z\in L^2(0,T;H_0^1(\Omega)),$ let $(w^z,\phi_1^z,\phi_2^z)$ be the solution of problem \eqref{eq4.2.2} associated with $z.$ Consider the adjoint problem of the linearized system of problem \eqref{eq4.2.2}:
\begin{equation}\label{eq4.2.3}
\begin{cases}
-\frac{\partial \psi^z}{\partial t}+\Delta^2\psi^z+a(x,t)\psi^z-\nabla\cdot(B(x,t)\psi^z)=G_1(z,\nabla z)\psi^z-\nabla\cdot(G_2(z,\nabla z)\psi^z)\\
+\sum_{i=1}^2\alpha_i\eta^z_i\chi_{\mathcal{O}_{i,d}},\,\,\,\,(x,t)\in Q,\\ 
\frac{\partial \eta^z_i}{\partial t}+\Delta^2\eta^z_i+a(x,t)\eta^z_i+B(x,t)\cdot\nabla \eta^z_i=F_u(\bar{u}+z,\nabla \bar{u}+\nabla z)\eta^z_i\\
+\nabla_pF(\bar{u}+z,\nabla \bar{u}+\nabla z)\cdot\nabla\eta^z_i-\frac{\psi^z}{\mu_i}\chi_{\mathcal{O}_i},\,\,\,\,(x,t)\in Q,\\
\psi^z=\frac{\partial \psi^z}{\partial\vec{n}}=0,\,\,\,\eta^z_i=\frac{\partial\eta^z_i}{\partial\vec{n}}=0,\,\,(x,t)\in\Sigma,\\
\psi^z(x,T)=\psi_0(x),\,\,\eta^z_i(x,0)=0,\,\,\,x\in\Omega.
\end{cases}
\end{equation}
Combining problem \eqref{eq4.2.2} with problem \eqref{eq4.2.3}, we obtain
\begin{align*}
\int_0^T\int_\mathcal{O}f\psi^z\,dxdt=\int_\Omega w^z(x,T)\psi_0(x)-w_0(x)\psi^z(x,0)\,dx+\int_Q\sum_{i=1}^2\alpha_i\eta^z_i\chi_{\mathcal{O}_{i,d}}w_{i,d}\,dxdt,
\end{align*}
which entails that problem \eqref{eq4.2.2} is null controllability if and only if 
\begin{align}\label{eq4.2.4}
\int_0^T\int_\mathcal{O}f\psi^z\,dxdt=\int_Q\sum_{i=1}^2\alpha_i\eta^z_i\chi_{\mathcal{O}_{i,d}}w_{i,d}\,dxdt-\int_\Omega w_0(x)\psi^z(x,0)\,dx
\end{align}
for any $\psi_0\in L^2(\Omega).$

By the similar proof of Theorem \ref{th1.1}, we conclude that for any $z\in L^2(0,T;H_0^1(\Omega)),$ there exists a unique leader control $f_z\in L^2(\mathcal{O}\times (0,T)),$ such that $w^z(T)\equiv 0.$ Moreover, there exists a positive constant $C$ independent of $z,$ such that
\begin{align*}
\int_0^T\int_{\mathcal{O}}|f_z|^2\,dxdt\leq C\left(\|w_0\|_{L^2(\Omega)}^2+\sum_{i=1}^2\int_0^T\int_{\mathcal{O}_{i,d}}\theta^{-2}(x,t)|w_{i,d}(x,t)|^2\,dxdt\right).
\end{align*}
Applying the Leray-Schauder's fixed points Theorem, we can deduce that for any $u_0\in L^2(\Omega),$ there exist at least one control $f\in L^2(\mathcal{O}\times (0,T))$ such that the corresponding solutions to problem \eqref{eq4.2.1} satisfy $w(T)\equiv 0.$ The details of proof is very similar with the proof of Theorem \ref{th3.2}, we omit it here.
\end{proof}
\subsection{Equilibria and quasi-equilibria}
The main aim of this subsection is to prove that there are some situations where the concepts of Nash equilibria and Nash quasi-equilibria are equivalent in the semi-linear case. Here, we state it as follows:
\begin{theorem}\label{th1.3}
Assume that $u_0\in L^2(\Omega)$ and $2\leq n\leq 20,$ $F\in W^{2,\infty}(\mathbb{R}^{n+1};\mathbb{R}),$ the ratio $\frac{\max\{\mu_1,\mu_2\}}{\min\{\alpha_1,\alpha_2\}}$ is sufficiently large and $\zeta_{i,d}\in L^{\infty}(\mathcal{O}_{i,d}\times (0,T))$ for $i=1,2.$ For any $f\in L^2(\mathcal{O}\times (0,T)),$ if there exists a positive constant $C,$ only depending on $\Omega,$ $\mathcal{O},$ $\mathcal{O}_i,$ $\mathcal{O}_{i,d},$ $T$ such that 
\begin{align}\label{1.12}
\frac{\max\{\mu_1,\mu_2\}}{\min\{\alpha_1,\alpha_2\}}\geq C(1+\|f\|_{L^2(\mathcal{O}\times (0,T))}),
\end{align}
then the couple $(v_1,v_2)$ fulfills \eqref{1.4} if and only if it fulfills equality \eqref{1.11}.
\end{theorem}
\begin{proof}
Let $f\in L^2(\mathcal{O}\times(0,T))$ be given and let $(v_1,v_2)$ be the corresponding Nash quasi-equilibrium pair. For any fixed $w_1\in\mathcal{H}_1$ and any $s\in\mathbb{R},$ denote by $u^s$ the solution of the following problem
\begin{equation}\label{eq4.3.1}
\begin{cases}
\frac{\partial u^s}{\partial t}+\Delta^2u^s+a(x,t)u^s+B(x,t)\cdot\nabla u^s=F(u^s,\nabla u^s)\\+f\chi_{\mathcal{O}}+(v_1+sw_1)\chi_{\mathcal{O}_1}+v_2\chi_{\mathcal{O}_2},\,\,\,\,(x,t)\in Q,\\
u^s=\frac{\partial u^s}{\partial\vec{n}}=0,\,\,\,\,\,(x,t)\in\Sigma,\\
u^s(x,0)=u_0(x),\,\,\,\,\,x\in\Omega
\end{cases}
\end{equation}
and let $u=u^s|_{s=0}.$ In what follows, we will estimate $D_1^2J_1(f;v_1,v_2)(w_1,w_1)$ from below.

Let $w^s$ be the solution of problem
\begin{equation}\label{eq4.3.2}
\begin{cases}
\frac{\partial w^s}{\partial t}+\Delta^2w^s+a(x,t)w^s+B(x,t)\cdot\nabla w^s\\
=F_u(u^s,\nabla u^s)w^s+\nabla_pF(u^s,\nabla u^s)\cdot\nabla w^s+\varphi_1\chi_{\mathcal{O}_1},\,\,\,\,(x,t)\in Q,\\
w^s=\frac{\partial w^s}{\partial\vec{n}}=0,\,\,\,\,\,(x,t)\in\Sigma,\\
w^s(x,0)=0,\,\,\,\,\,x\in\Omega
\end{cases}
\end{equation}
and denote by $w=w^s|_{s=0}.$ Thanks to
\begin{align*}
D_1J_1(f;v_1+sw_1,v_2)\varphi_1=\alpha_1\int_0^T\int_{\mathcal{O}_{1,d}}(u^s-\zeta_{1,d})w^s\,dxdt+\mu_1\int_0^T\int_{\mathcal{O}_1}(v_1+sw_1)\varphi_1
\end{align*}
for any $\varphi_1\in\mathcal{H}_1,$ we conclude that
\begin{align*}
&D_1J_1(f;v_1+sw_1,v_2)\varphi_1-D_1J_1(f;v_1,v_2)\varphi_1\\
=&\alpha_1\int_0^T\int_{\mathcal{O}_{1,d}}(u^s-\zeta_{1,d})w^s\,dxdt-\alpha_1\int_0^T\int_{\mathcal{O}_{1,d}}(u-\zeta_{1,d})w\,dxdt+s\mu_1\int_0^T\int_{\mathcal{O}_1}w_1\varphi_1.
\end{align*}
Let us consider the adjoint system of problem \eqref{eq4.3.2}:
\begin{equation}\label{eq4.3.3}
\begin{cases}
-\frac{\partial \phi^s}{\partial t}+\Delta^2\phi^s+a(x,t)\phi^s-\nabla\cdot(B(x,t)\phi^s)=F_u(u^s,\nabla u^s)\phi^s\\
-\nabla\cdot(\nabla_pF(u^s,\nabla u^s)\phi^s)+\alpha_1\chi_{\mathcal{O}_{1,d}}(u^s-\zeta_{1,d}),\,\,\,\,(x,t)\in Q,\\
\phi^s=\frac{\partial \phi^s}{\partial\vec{n}}=0,\,\,\,\,\,(x,t)\in\Sigma,\\
\phi^s(x,T)=0,\,\,\,\,\,x\in\Omega
\end{cases}
\end{equation}
and denote by $\phi=\phi^s|_{s=0},$ we deduce from problem \eqref{eq4.3.2} and problem \eqref{eq4.3.3} that
\begin{align*}
\alpha_1\int_0^T\int_{\mathcal{O}_{1,d}}(u^s-\zeta_{1,d})w^s\,dxdt=\int_0^T\int_{\mathcal{O}_1}\varphi_1\phi^s\,dxdt.
\end{align*}
Thus, we obtain
\begin{align*}
&D_1J_1(f;v_1+sw_1,v_2)\varphi_1-D_1J_1(f;v_1,v_2)\varphi_1\\
=&\int_0^T\int_{\mathcal{O}_1}\varphi_1(\phi^s-\phi)\,dxdt+s\mu_1\int_0^T\int_{\mathcal{O}_1}w_1\varphi_1.
\end{align*}
It is easy to prove that the following limits exist:
\begin{align*}
h=\lim_{s\rightarrow 0}\frac{u^s-u}{s},\,\,\,\eta=\lim_{s\rightarrow 0}\frac{\phi^s-\phi}{s}\end{align*}
and $(h,\eta)$ is the solution of the following problem
\begin{equation}\label{eq4.3.4}
\begin{cases}
\frac{\partial h}{\partial t}+\Delta^2h+a(x,t)h+B(x,t)\cdot\nabla h=F_u(u,\nabla u)h+\nabla_pF(u,\nabla u)\cdot\nabla h+w_1\chi_{\mathcal{O}_1},\,\,\,\,(x,t)\in Q,\\
-\frac{\partial \eta}{\partial t}+\Delta^2\eta+a(x,t)\eta-\nabla\cdot(B(x,t)\eta)=F_{uu}(u,\nabla u)\phi h+\nabla_pF_u(u,\nabla u)\cdot\nabla h\phi+F_u(u,\nabla u)\eta\\
-\nabla\cdot(\nabla_pF_u(u,\nabla u)h\phi+\nabla_p^2F(u,\nabla u)\nabla h\phi+\nabla_pF(u,\nabla u)\eta)+\alpha_1\chi_{\mathcal{O}_{1,d}}h,\,\,\,\,(x,t)\in Q,\\
h=\frac{\partial h}{\partial\vec{n}}=0,\,\,\,\eta=\frac{\partial \eta}{\partial\vec{n}}=0,\,\,(x,t)\in\Sigma,\\
h(x,0)=0,\,\,\eta(x,T)=0,\,\,\,x\in\Omega.
\end{cases}
\end{equation}
Moreover, we obtain
\begin{align*}
D_1^2J_1(f;v_1,v_2)(w_1,\varphi_1)=\int_0^T\int_{\mathcal{O}_1}\varphi_1\eta\,dxdt+\mu_1\int_0^T\int_{\mathcal{O}_1}w_1\varphi_1\,dxdt.
\end{align*}
In particular, we have
\begin{align*}
D_1^2J_1(f;v_1,v_2)(w_1,w_1)=\int_0^T\int_{\mathcal{O}_1}w_1\eta\,dxdt+\mu_1\int_0^T\int_{\mathcal{O}_1}|w_1|^2\,dxdt.
\end{align*}
From \eqref{eq4.3.4}, we deduce that
\begin{align}\label{eq4.3.5}
\nonumber\int_0^T\int_{\mathcal{O}_1}w_1\eta\,dxdt=&\int_QF_{uu}(u,\nabla u)\phi h^2+\nabla_pF_u(u,\nabla u)\cdot\nabla h\phi h+\alpha_1\chi_{\mathcal{O}_{1,d}}|h|^2\,dxdt\\
&+\int_Q\left(\nabla_pF_u(u,\nabla u)h\phi+\nabla_p^2F(u,\nabla u)\nabla h\phi\right)\cdot\nabla h\,dxdt.
\end{align}
In what follows, we will estimate the right-hand side of \eqref{eq4.3.5}. To this purpose, we will prove that 
\begin{align*}
\phi\in L^r(0,T;L^s(\Omega)),\,\,\,h\in L^{2r'}(0,T;W^{1,2s'}(\Omega)),
\end{align*}
where $r'$ and $s'$ are the conjugate of $r$ and $s,$ respectively.

From the regularity theorem of parabolic system, we can easily deduce that 
\begin{align*}
h\in L^{\infty}(0,T;H_0^2(\Omega))\cap H^1(0,T;L^2(\Omega))\cap L^2(0,T;H^4(\Omega))
\end{align*}
and
\begin{align}\label{eq4.3.6}
\|h_t\|_{L^2(Q)}+\|h\|_{L^{\infty}(0,T;H_0^2(\Omega))\cap L^2(0,T;H^4(\Omega))}\leq C\|w_1\|_{\mathcal{H}_1}.
\end{align}
{\bf Case 1: $n>4.$}

Applying interpolation inequality $H^1(0,T;L^2(\Omega))\cap L^2(0,T;H^4(\Omega))\subset L^a(0,T;L^b(\Omega)),$ yields
\begin{align*}
h\in L^a(0,T;W^{1,b}(\Omega)),
\end{align*}
where
\begin{align*}
\frac{1}{a}=\frac{1}{2}-\theta,\,\,\,\textit{for\,\,some}\,\,\theta\in(0,\frac{1}{2}),\,\,\frac{1}{b}=\frac{a(n-2)-8}{2an}.
\end{align*}
Let $a=2r'$ and $b=2s',$ we obtain
\begin{align*}
r=\frac{a}{a-2},\,\,\,s=\frac{an}{2a+8}.
\end{align*}
Since $u\in L^{\infty}(0,T;L^2(\Omega))\cap H^1(0,T;H^{-2}(\Omega))\cap L^2(0,T;H_0^2(\Omega)),$ we deduce from interpolation inequality that $u\in L^p(0,T;L^q(\Omega)),$ where
\begin{align*}
\frac{1}{p}=\frac{\bar{\theta}}{2},\,\,\,\textit{for\,\,some}\,\,\bar{\theta}\in(0,1),\,\,\frac{1}{q}=\frac{a(n-4)+8}{2an}.
\end{align*}
It follows from the regularity theory of parabolic equation that
\begin{align*}
\phi\in L^p(0,T;W^{2,q}(\Omega))\subset L^p(0,T;L^{\frac{qn}{n-2q}})=L^p(0,T;L^{\frac{2pn}{p(n-8)+8}}(\Omega)).
\end{align*}
Taking $p=r,$ we obtain $\phi\in L^r(0,T;L^{\frac{2an}{an-16}}(\Omega)).$ Thus, in order to $\phi\in L^r(0,T;L^s(\Omega)),$ we need to require that
\begin{align*}
\frac{an}{2a+8}\leq\frac{2an}{an-16},
\end{align*}
which is true if and only if $n<20.$ In this case, we obtain
\begin{align*}
\left|\int_0^T\int_{\mathcal{O}_1}w_1\eta\,dxdt\right|\leq&C\int_Q|\phi| |h|^2+|\nabla h||\phi| |h|+|\nabla h|^2|\phi|+\alpha_1\chi_{\mathcal{O}_{1,d}}|h|^2\,dxdt\\
\leq&C\|\phi\|_{L^r(0,T;L^s(\Omega))}\|\nabla h\|_{L^{2r'}(0,T;L^{2s'}(\Omega))}^2+C\alpha_1\|h\|_{L^2(Q)}^2\\
\leq&C\alpha_1(1+\|u-\zeta_{1,d}\|_{L^2(\mathcal{O}_{1,d}\times(0,T))})\|w_1\|_{\mathcal{H}_1}^2\\
\leq&C\alpha_1(1+\|f\|_{L^2(\mathcal{O}\times(0,T))})\|w_1\|_{\mathcal{H}_1}^2.
\end{align*}
{\bf Case 2: $2\leq n\leq 4.$}
\begin{align*}
\left|\int_0^T\int_{\mathcal{O}_1}w_1\eta\,dxdt\right|\leq&C\int_Q|\phi| |h|^2+|\nabla h||\phi| |h|+|\nabla h|^2|\phi|+\alpha_1\chi_{\mathcal{O}_{1,d}}|h|^2\,dxdt\\
\leq&C\int_0^T\|\phi(t)\|_{L^n(\Omega)}\|\nabla h(t)\|_{L^{\frac{2n}{n-2}}(\Omega)}^2\,dt+C\alpha_1\|h\|_{L^2(Q)}^2\\
\leq&C\|\phi\|_{L^2(0,T;H^2_0(\Omega))}\|h\|_{L^{\infty}(0,T;H_0^2(\Omega)}^2+C\alpha_1\|h\|_{L^2(Q)}^2\\
\leq&C\alpha_1(1+\|u-\zeta_{1,d}\|_{L^2(\mathcal{O}_{1,d}\times(0,T))})\|w_1\|_{\mathcal{H}_1}^2\\
\leq&C\alpha_1(1+\|f\|_{L^2(\mathcal{O}\times(0,T))})\|w_1\|_{\mathcal{H}_1}^2.
\end{align*}
Thus, if $2\leq n<20,$ we have
\begin{align*}
\left|\int_0^T\int_{\mathcal{O}_1}w_1\eta\,dxdt\right|\leq C\alpha_1(1+\|f\|_{L^2(\mathcal{O}\times(0,T))})\|w_1\|_{\mathcal{H}_1}^2,
\end{align*}
which entails that there exists a generic positive constant $C_1,$ such that
\begin{align*}
D_1^2J_1(f;v_1,v_2)(w_1,w_1)\geq \left(\mu_1-C_1\alpha_1(1+\|f\|_{L^2(\mathcal{O}\times(0,T))})\right)\|w_1\|_{\mathcal{H}_1}^2.
\end{align*}
Similarly, we can also prove that there exists a generic positive constant $C_2,$ such that
\begin{align*}
D_2^2J_2(f;v_1,v_2)(w_2,w_2)\geq \left(\mu_2-C_2\alpha_2(1+\|f\|_{L^2(\mathcal{O}\times(0,T))})\right)\|w_2\|_{\mathcal{H}_2}^2.
\end{align*}
Therefore, if $\frac{\max\{\alpha_1,\alpha_2\}}{\min\{\mu_1,\mu_2\}}$ is sufficiently small, such that $\mu_i-C_i\alpha_i(1+\|f\|_{L^2(\mathcal{O}\times(0,T))})$ is positive, we conclude that $v_1$ minimizes $J_1(f;\cdot,\hat{v}_2)$ and $v_2$ minimizes $J_2(f;\hat{v}_1,\cdot),$ which entails that the pair $(v_1,v_2)$ is in fact a Nash equilibrium.
\end{proof}
\section*{Acknowledgement}
This work was supported by the National Science Foundation of China Grant (11801427, 11871389) and the Fundamental Research Funds for the Central Universities (xzy012022008, JB210714).

\bibliographystyle{abbrv}
\bibliography{BIB}
\end{document}